\documentclass[a4paper]{article}
\usepackage{amsmath}
\usepackage{amsthm}
\usepackage{amsfonts}
\usepackage{amssymb}
\usepackage{bussproofs}
\usepackage{mathtools}
\usepackage{verbatim}
\usepackage{dsfont}
\usepackage{mathabx}
\usepackage[all, 2cell]{xy}
\usepackage[all]{xy}
\usepackage{wasysym}
\usepackage{rotating}
\usepackage{geometry}
\usepackage{trfsigns}
\usepackage{cmll}
\usepackage{authblk}
\usepackage{hyperref}
\usepackage{cleveref}
\usepackage{lipsum}
\usepackage{extpfeil}
\usepackage{soul}
\usepackage{mathabx}
\usepackage{stmaryrd}
\usepackage{listings}

\usepackage[scr]{rsfso}

\theoremstyle{defin}
\newtheorem{defin}{Definition}

\theoremstyle{theorem}
\newtheorem{theorem}{Theorem}

\theoremstyle{prop}
\newtheorem{prop}{Proposition}

\theoremstyle{lemma}
\newtheorem{lemma}{Lemma}

\theoremstyle{ex}

\theoremstyle{col}

\providecommand{\keywords}[1]
{
  \small
  \textbf{\textit{Keywords---}} #1
}

\author[1,2]{Daniel Rogozin}
\affil[1]{Institute for Information Transmission Problems, RAS, Moscow, Russia}
\affil[2]{Serokell O\"{U}, Tallinn, Estonia}
\date{}
\title{Categorical and Algebraic Aspects of the Intuitionistic Modal Logic $\operatorname{IEL}^{-}$ and its Predicate Extensions}

\begin{document}

\maketitle

\begin{abstract}
The system of intuitionistic modal logic ${\bf IEL}^{-}$ was proposed by S. Artemov and T. Protopopescu as the intuitionistic version of belief logic \cite{Artemov}. We construct the modal lambda calculus which is Curry-Howard isomorphic to ${\bf IEL}^{-}$ as the type-theoretical representation of applicative computation widely known in functional programming.
We also provide a categorical interpretation of this modal lambda calculus considering coalgebras associated with a monoidal functor on a cartesian closed category.

Finally, we study Heyting algebras and locales with corresponding operators. Such operators are used in point-free topology as well. We study compelete Kripke-Joyal-style semantics for predicate extensions of ${\bf IEL}^{-}$ and related logics using Dedekind-MacNeille completions and modal cover systems introduced by Goldblatt \cite{goldblatt2011cover}. The paper extends the conference paper published in the LFCS'20 volume \cite{rogozin2020modal}.
\end{abstract}

\keywords{Intuitionistic modal logic, Modal type theory, Functional programming, Locales, Prenucleus,
Cover systems}

\section{Introduction}

Intuitionistic modal logic study extensions of intuitionistic logic with modal operators. One may consider such extensions from two directions. The first perspective is a consideration of intuitionistic modal logic as a branch of modal logic. Here, intuitionistic modalities might be interpreted as a constructive necessity, provability in Heyting arithmetics, intuitionistic knowledge, and so on. The second perspective is type-theoretic. This provides a more computational interpretation of intuitionistic modalities. Each value in an arbitrary computation is annotated with the relevant data type.

The first perspective arises to Prior, who introduced the system called ${\bf MIPC}$ \cite{prior2003time} to investigate modal counterparts of intuitionistic monadic logic. Bull \cite{bull1997mipc} and Ono \cite{ono1977some} studies the relation between intuitionistic modalities and quantifiers. Bezhanishvili examined quite comprehensively monadic Heyting algebras, see, for instance, \cite{bezhanishvili1998varieties}.

Fischer-Servi provided an intuitionistic analogue of the minimal normal modal logic containing $\Box$ and $\Diamond$ as mutually inexpressible connectives \cite{servi1977modal}.

Williamson \cite{williamson1992intuitionistic} discussed the question of intuitionistic epistemic modalities considering the problem of an intuitionist knowledge by means of the capability of verification. Artemov and Protopopescu developed this direction further, see \cite{Artemov} and \cite{protopopescu2015intuitionistic}.

We also emphasise briefly the direction related to Heyting algebras with operators. Heyting algebras with Fischer-Servi modal operators have a topological duality piggybacked on Esakia's results \cite{esakia2019heyting}. This duality provides the explicit description of general descriptive frames for extensions of intuitionistic modal logic containing the Fischer-Servi logic, see the paper by Palmigiano \cite{palmigiano2004dualities}. Macnab examined the class of Heyting algebras nuclei \cite{macnab1981modal}. We discuss nuclei closely in Section 5. Here we merely claim that logic of Heyting algebras with nuclei and their predicate extensions was investigated by Bezhanishvili and Ghilardi \cite{bezhanishvili2007algebraic}; Goldblatt \cite{goldblatt1981grothendieck} \cite{goldblatt2011cover}; Fairtlough, Mendler, and Walton \cite{fairtlough1997propositional} \cite{fairtlough1997quantified}.

We refer the reader to this paper by Wolter and Zakharyaschev \cite{wolter1999intuitionistic}, the paper by Bo{\v{z}}i{\'c} and Do{\v{s}}en \cite{bovzic1984models}, and the monograph by Simpson \cite{simpson1994proof}. These works contain the underlying results in model-theoretic aspects of intuitionistic modal logic.

The second perspective we emphasise above is related to intuitionistic modalities in a computational landscape. The Curry-Howard correspondence provides bridges between intuitionistic proofs and programs understood in a type-theoretical sense \cite{Neder}.

Modal lambda calculi often correspond to certain intuitionistic modal logics, see the papers by Artemov \cite{artemov2003embedding}; Bierman and de Paiva \cite{bierman2000intuitionistic}; Davies and Pfenning \cite{davies2001modal}; Fairtlough and Mendler\cite{fairtlough2000logical}, etc. De Paiva and Ritter \cite{DEPAIVA2016143} provided a categorical framework for such type theories.

One may motivate modal types with functional programming. Let us observe a sort of computation called monadic. A monad is a concept in functional programming implemented in the functional language called Haskell. Moggi examined monads type-theoretically \cite{moggi1991notions}. Very informally, a monad is a method of structuring computations as linearly connected chain of actions within such types as the list or the input/output (\verb"IO"). Such sequences are often called \emph{pipelines} in which one passes a value from an external world and yield a result after some series of actions. There is a way to consider computational monads logically within intuitionistic modal logic.

Functional languages such as Haskell have so-called type classes \footnote{In Haskell, type class is a general interface for some special group of data types consising of a group of methods for those types.} for computation within an environment. By \emph{computational context} (or, \emph{environment}), we mean some, roughly speaking, type-level map $f$, where $f$ is a
``function'' from $*$ to $*$: such a type-level map takes a simple type which has kind $*$ and yields another simple type of kind $*$. See \cite{vytiniotis2013evidence} and \cite{weirich2013system} to have a more detailed description of Haskell-like type systems.

Here, the underlying type class is \verb"Functor" which has the following formal definition:

\begin{lstlisting}[language=Haskell]
  class Functor f where
    fmap :: (a -> b) -> f a -> f b
\end{lstlisting}

\verb"Functor" is a generalisation of higher-order functions as \verb"map" thay merely yields an image of a list by a given function. Generally, \verb"Functor" provides a uniform method to carry unary functions through parametrised types. In other words, the notion of a functor in functional programming is inspired by categorical functor to a certain degree.

One may extend a functor to the so-called monad a bit inspired by categorical Kleisli triples. In Haskell-like languages, one also has the type class called \verb"Monad", a type class of an abstract data type of action in some computational environment. Here we define the \verb"Monad" type class as follows:

\begin{lstlisting}[language=Haskell]
  class Functor m => Monad m where
    return :: a -> m a
    (>>=)  :: m a -> (a -> m b) -> m b
\end{lstlisting}

\verb"Monad" is a type class that extends \verb"Functor" with two methods called \verb"return" and \verb"(>>=)" (a monadic bind). Monads present a uniform technique for various computations such as computation with a mutable state, many-valued computation, side effect input-output computation, etc. All those kinds are arranged as pipelines. Historically, monads were implemented in Haskell to process side-effects that arise in the input/output world. The advantage of a monad is an ability to isolate side-effects within a monad remaining code functional. That is, one has a tool to describe a sequence of actions, where the result of each step depends on previous ones. In other words, one has so-called the monadic binding.

Monadic metalanguage is the modal lambda calculus that describes a computation within an abstract monad \cite{moggi1991notions}. More formally, this is a modal extension of the simply-typed lambda calculus Curry-Howard isomorphic to the lax logic, the logic of Heyting algebras with a nucleus operator that we discussed above. The typing rules for modalities of this metalanguage correspond to the \verb"return" and the monadic bind methods.

Let us take a look at the example of a monad. There is the parametrised data type \verb"Maybe" in Haskell. \verb"Maybe" allows one to redefine a partial function as total:
\begin{lstlisting}[language=Haskell]
data Maybe a = Nothing | Just a
\end{lstlisting}
This data type consists of two constructors. Suppose we deal with some computation that might terminate with a failure. \verb"Nothing" is a flag that claims this failure arose. The second constructor \verb"Just" stores some value of $a$, that is, a result of a computation that terminated successfully.

The \verb"Maybe" instance of \verb"Monad" is the following one:
\begin{lstlisting}[language=Haskell]
instance Monad Maybe where
  return = Just
  (Just v) >>= f = f v
  Nothing  >>= f = Nothing
\end{lstlisting}

Here, the \verb"return" method merely embeds any value of $a$ into $\verb"Maybe" \: a$. The implementation of a monadic bind for \verb"Maybe" is also quite simple. Suppose one has a function $f$ of type $a \to \verb"Maybe " b$ and some value $x$ of type $\verb"Maybe " a$. Here we match on $x$. If $x$ is \verb"Nothing", then the monadic bind yields \verb"Nothing". Otherwise, we extract the value of type $a$ and apply a given function.

The monad interface for \verb"Maybe" allows one to perform sequences of actions, where some values might be undefined. If all values are well defined on each step, then the result of an execution is a term of the form \verb"Just v". Otherwise, if something went wrong and we have no required value somewhere, then the computation terminates with \verb"Nothing". The other examples of \verb"Monad" instances have more or less the same explanation.

Let us discuss the \verb"Applicative" class. Paterson and McBride proposed this class to describe effectful applicative programming \cite{McP}. One may consider the \verb"Applicative" type class as an intermediate one between \verb"Functor" and \verb"Monad". See this paper to have a more precise understanding of the connection between applicative functors and monads \cite{lindley2011idioms}.

Here is the precise definition of \verb"Applicative":

\begin{lstlisting}[language=Haskell]
class Functor f => Applicative f where
  pure :: a -> f a
  (<*>) :: f (a -> b) -> f a -> f b
\end{lstlisting}

The main aim of an applicative functor to generalise functor for functions of an arbitrary arity, for instance:
\begin{lstlisting}[language=Haskell]
liftA2
  :: Applicative f
  => (a -> b -> c)
  -> f a -> f b -> f c
liftA2 f x y = ((pure f) <*> x) <*> y
\end{lstlisting}

\verb"liftA2" is a version of \verb"fmap" for an arbitrary two-argument function. It is clear that one may implement \verb"liftA3", \verb"liftA4", and \verb"liftAn" for each $n < \omega$.

In this paper, we consider applicative computation type-theoretically. The modal axioms of ${\bf IEL}^{-}$ and types of the \verb"Applicative" methods in Haskell-like languages are quite similar to each other. We investigate the relationship between intuitionistic epistemic logic ${\bf IEL}^{-}$ and applicative computation providing the modal lambda calculus Curry-Howard isomorphic to ${\bf IEL}^{-}$.

This calculus consists of the rules for the simply-typed lambda-calculus extended with the modal rules. We assume that the proposed modal lambda calculus axiomatises applicative computation. We provide a proof-theoretical view of this sort of computation and prove such metatheoretical properties as strong normalisation and confluence. The idea to consider applicative functors type-theoretically belongs to Krishnaswami \cite{Krishna}. We are going to develop his ideas considering the ${\bf IEL}^{-}$ from a computational perspective. Litak and coauthors \cite{litak2017negative} observed that the logic ${\bf IEL}^{-}$ might be treated as a logic of an applicative functor as well \footnote{John Connor (the City College of New York) also connected the intuitionistic epistemic logic ${\bf IEL}^{-}$ with propositional truncation in Homotopy Type Theory. He presented his results at the category theory seminar, the CUNY Graduate Centre. At the moment, there is only a video of that talk on YouTube.}.

We also analyse semantic aspects of ${\bf IEL}^{-}$ and related logics. We study categorical semantics for the provided modal lambda calculus and cover semantics for quantified versions of intuitionistic modal logic with ${\bf IEL}^{-}$-like modalities.

In other words, we study the logic ${\bf IEL}^{-}$ and its relatives from several perspectives. In this paper, we examine computational aspects of these logics as well as algebraic and semantic ones.

\section{The modal type theory based on ${\bf IEL}^{-}$}

Intuitionistic modal logic ${\bf IEL}^{-}$ was proposed by S. Artemov and T. Protopopescu \cite{Artemov}. According to the authors, ${\bf IEL}^{-}$ represents beliefs agreed with BHK-semantics of intuitionistic logic. ${\bf IEL}^{-}$ is a weaker version of the system ${\bf IEL}$ that formalises knowledge as provably consistent intuitionistic belief. This logic consists of the following axioms and derivation rules:
\begin{defin} Intuitionistic epistemic logic ${\bf IEL}^{-}$:

\begin{enumerate}
  \item $(\varphi \to (\psi \to \theta)) \to ((\varphi \to \psi) \to (\varphi \to \theta))$
  \item $\varphi \to (\psi \to \varphi)$
  \item $\varphi \to (\psi \to (\varphi \land \psi))$
  \item $\varphi_1 \land \varphi_2 \to \varphi_i$, $i = 1,2$
  \item $(\varphi \to \theta) \to ((\psi \to \theta) \to (\varphi \lor \psi \to \theta))$
  \item $\varphi_i \to \varphi_1 \lor \varphi_2$, $i = 1,2$
  \item $\bot \to \varphi$
  \item $\bigcirc (\varphi \to \psi) \to (\bigcirc \varphi \to \bigcirc \psi)$
  \item $\varphi \to \bigcirc \psi$
  \item From $\varphi \to \psi$ and $\varphi$ infer $\psi$ (Modus ponens).
\end{enumerate}
\end{defin}

The $\varphi \to \bigcirc \psi$ axiom is also called \emph{co-reflection}. One may consider this axiom as the principle that connects intuitionistic truth and intuitionistic knowledge. From a Kripkean point, ${\bf IEL}^{-}$ is the logic of all frames $\langle W, \leq, E \rangle$, where $\langle W, \leq \rangle$ is a partial order and $E$ is a binary ``knowledge'' relation, a subrelation of $\leq$. The relation $E$ obeys the following conditions:

\begin{enumerate}
  \item $E(w) \subseteq {\uparrow w}$ for each $w \in W$.
  \item $E(u) \subseteq E(w)$, if $w R u$.
\end{enumerate}

A model for ${\bf IEL}^{-}$ is a quadruple $\mathcal{M} = \langle W, \leq, E, \vartheta \rangle$, an extended intuitionistic Kripke model with the additional forcing relation for modal formulas defined via the relation $E$. As usual, $\vartheta$ is a valuation that maps evepy propositional letter to some upper cone.

The $\bigcirc$ connective has the ``necessity'' semantics:

\begin{center}
$\mathcal{M}, x \Vdash \bigcirc \varphi \Leftrightarrow \forall y \in E(x) \:\: \mathcal{M}, y \Vdash \varphi$.
\end{center}

${\bf IEL}$, the full epistemic intuitionistic logic, extends ${\bf IEL}^{-}$ as ${\bf IEL} = {\bf IEL}^{-} \oplus \bigcirc \varphi \to \neg \neg \varphi$. This additional axiom is often called the \emph{intuitionistic relfection principle}. An ${\bf IEL}$-frame is an ${\bf IEL}^{-}$ frame with the condition $E(u) \neq \emptyset$ for each $u \in W$. That is, a knowledge relation is serial.

Artemov and Protopopescu proved that these logics are Kripke-complete with the standard construction with the canonical model on prime theories \cite{Artemov}.

V. Krupski and A. Yatmanov investigated proof-theoretical and algorithmic aspects of the logic ${\bf IEL}$. In this paper \cite{krupski2016sequent}, they provided the sequent calculus for ${\bf IEL}$ and proved that the derivability problem of this calculus is PSPACE-complete. ${\bf IEL}^{-}$ is decidable as well since this logic has the finite model property, see the paper by Wolter and Zakharyaschev \cite{wolter1997intuitionistic}.

For further purposes, we define the natural deduction calculus for ${\bf IEL}^{-}$ that we call ${\bf NIEL}^{-}$. For simplicity, we restrict our language to $\to$, $\land$, and $\bigcirc$.

\begin{defin} The natural deduction calculus {\bf NIEL}$^{-}$ for ${\bf IEL}^{-}$ is an extension of the intuitionistic natural deduction calculus with the additional inference rules for modality:

\begin{center}
  \begin{prooftree}
    \AxiomC{$ $}
    \RightLabel{\bf{ax}}
    \UnaryInfC{$\Gamma, \varphi \vdash \varphi$}
  \end{prooftree}
\end{center}

\begin{minipage}{0.5\textwidth}
  \begin{flushleft}
  \begin{prooftree}
    \AxiomC{$\Gamma, \varphi \vdash \psi$}
    \RightLabel{$\to_I$}
    \UnaryInfC{$\Gamma \vdash \varphi \to \psi$}
  \end{prooftree}
  \begin{prooftree}
    \AxiomC{$\Gamma \vdash \varphi$}
    \AxiomC{$\Gamma \vdash \psi$}
    \RightLabel{$\land_I$}
    \BinaryInfC{$\Gamma \vdash \varphi \land \psi$}
  \end{prooftree}
  \begin{prooftree}
    \AxiomC{$\Gamma \vdash \varphi$}
    \RightLabel{${\bigcirc_I}_1$}
    \UnaryInfC{$\Gamma \vdash \bigcirc \varphi$}
\end{prooftree}
  \end{flushleft}
\end{minipage}
\begin{minipage}{0.5\textwidth}
  \begin{flushright}
  \begin{prooftree}
    \AxiomC{$\Gamma \vdash \varphi \to \psi$}
    \AxiomC{$\Gamma \vdash \varphi$}
    \RightLabel{$\to_E$}
    \BinaryInfC{$\Gamma \vdash \psi$}
  \end{prooftree}
  \begin{prooftree}
    \AxiomC{$\Gamma \vdash \varphi_1 \land \varphi_2$}
    \RightLabel{$\land_E, i = 1,2$}
    \UnaryInfC{$\Gamma \vdash \varphi_i$}
  \end{prooftree}
  \begin{prooftree}
  \AxiomC{$\Gamma \vdash \bigcirc \overrightarrow{\varphi}$}
  \AxiomC{$\overrightarrow{\varphi} \vdash \psi$}
  \RightLabel{${\bigcirc_I}_2$}
  \BinaryInfC{$\Gamma \vdash \bigcirc \psi$}
  \end{prooftree}
  \end{flushright}
\end{minipage}
\end{defin}

The first modal rule allows one to derive co-reflection and its consequences. The second modal rule is a counterpart of $\bigcirc_I$ rule in natural deduction calculus for constructive ${\bf K}$ (see \cite{ModalLa}). We will denote $\Gamma \vdash \bigcirc \varphi_1, \dots, \Gamma \vdash \bigcirc \varphi_n$ and $\varphi_1,\dots,\varphi_n \vdash \psi$ as $\Gamma \vdash \bigcirc \overrightarrow{\varphi}$ and $\overrightarrow{\varphi} \vdash \psi$ respectively for brevity.

It is straightforward to check that the second modal rule is equivalent to the ${\bf K}$-like $\bigcirc$-rule:
\begin{prooftree}
  \AxiomC{$\Gamma \vdash \varphi$}
  \UnaryInfC{$\bigcirc \Gamma \vdash \bigcirc \varphi$}
\end{prooftree}

Let us show that one may translate ${\bf NIEL}^{-}$ into ${\bf IEL}^{-}$ as follows:

  \begin{lemma}
    $\Gamma \vdash_{{\bf NIEL}^{-}} \varphi \Rightarrow {\bf IEL}^{-}_{\to, \land, \bigcirc} \vdash \bigwedge \Gamma \rightarrow \varphi$.
  \end{lemma}

  \begin{proof}
Induction on the derivation. Let us consider the modal cases.

\begin{enumerate}
  \item If $\Gamma \vdash_{{\bf NIEL}^{-}} \varphi$, then ${\bf IEL}^{-}_{\to, \land, \bigcirc} \vdash \bigwedge \Gamma \rightarrow \bigcirc \varphi$.

  \vspace{\baselineskip}

$\begin{array}{lll}
(1) & \bigwedge \Gamma \rightarrow \varphi & \text{assumption}\\
(2) & \varphi \rightarrow \bigcirc \varphi &\text{co-reflection}\\
(3) & (\bigwedge \Gamma \rightarrow \varphi) \rightarrow ((\varphi \rightarrow \bigcirc \varphi) \rightarrow (\bigwedge \Gamma \rightarrow \bigcirc \varphi))&\text{IPC theorem}\\
(4) & (\varphi \rightarrow \bigcirc \varphi) \rightarrow (\bigwedge \Gamma \rightarrow \bigcirc \varphi) &\text{from (1), (3) and MP}\\
(5) & \bigwedge \Gamma \rightarrow \bigcirc \varphi &\text{from (2), (4) and MP}\\
\end{array}$

  \vspace{\baselineskip}

\item If $\Gamma \vdash_{{\bf NIEL}^{-}} \bigcirc \overrightarrow{\varphi}$ and $\overrightarrow{A} \vdash \psi$, then ${\bf IEL}^{-}_{\to, \land, \bigcirc} \vdash \bigwedge \Gamma \rightarrow \bigcirc \psi$.

  \vspace{\baselineskip}

$\begin{array}{lll}
(1) &\bigwedge \Gamma \rightarrow \bigcirc \varphi_1, \dots, \bigwedge \Gamma \rightarrow \bigcirc \varphi_n & \text{assumption} \\
(2) &\bigwedge \Gamma \rightarrow \bigwedge \limits_{i = 1}^{n} \bigcirc \varphi_i & \text{{\bf IEL}$^{-}$ theorem} \\
(3) &\bigwedge \limits_{i = 1}^{n} \bigcirc \varphi_i \rightarrow \bigcirc \bigwedge \limits_{i = 1}^{n} \varphi_i & \text{{\bf IEL}$^{-}$ theorem} \\
(4) &\bigwedge \Gamma \rightarrow \bigcirc \bigwedge \limits_{i = 1}^{n} \varphi_i & \text{from (2), (3) and transitivity} \\
(5) &\bigwedge \limits_{i = 1}^{n} \varphi_i \rightarrow \psi& \text{assumption} \\
(6) &(\bigwedge \limits_{i = 1}^{n} \varphi_i \rightarrow \psi) \rightarrow \bigcirc (\bigwedge \limits_{i = 1}^{n} \varphi_i \rightarrow \psi)& \text{co-reflection}\\
(7) &\bigcirc (\bigwedge \limits_{i = 1}^{n} \varphi_i \rightarrow \psi)& \text{from (5), (6) and MP} \\
(8) &\bigcirc \bigwedge \limits_{i = 1}^{n} \varphi_i \rightarrow \bigcirc \psi & \text{from (7) and normality} \\
(9) &\bigwedge \Gamma \rightarrow \bigcirc \psi & \text{from (4), (8) and transitivity}
\end{array}$
\end{enumerate}

  \end{proof}

\begin{lemma} If ${\bf IEL}^{-}_{\to, \land, \bigcirc} \vdash A$, then ${\bf NIEL}^{-} \vdash A$.
\end{lemma}

  \begin{proof}
    A straightforward derivation of modal axioms in ${\bf NIEL}^{-}$. We will consider those derivations via terms below.
  \end{proof}

One may enrich this natural deduction calculus with the well-known inference rules for disjunction and bottom and prove the same lemmas as above. We build further the typed lambda-calculus based on the ${\bf NIEL}^{-}$ by proof-assignment in the inference rules.

Let us define terms and types to introduce the modal typed lambda-calculus.

\begin{defin} The set of terms:

Let $\mathbb{V} = \{ x, y, z, \dots \}$ be the set of variables, the following grammar generates the set $\Lambda_{\bigcirc}$ of terms:

\begin{center}
$\Lambda_{\bigcirc} ::= \mathbb{V} \: | \:  (\lambda \mathbb{V}.\Lambda_{\bigcirc}) \: | \: (\Lambda_{\bigcirc} \Lambda_{\bigcirc}) \: | \: (\langle \Lambda_{\bigcirc} , \Lambda_{\bigcirc} \rangle) \: | \: (\pi_1 \Lambda_{\bigcirc}) \: | \: (\pi_2 \Lambda_{\bigcirc}) \: | \: ({\bf pure \: } \:\Lambda_{\bigcirc}) \: | \: ({\bf let \:} \bigcirc \mathbb{V}^{*} = \Lambda_{\bigcirc}^{*} \:\: {\bf in} \:\: \Lambda_{\bigcirc})$
\end{center}

\end{defin}
where $\mathbb{V}^{*}$ and $\Lambda_{\bigcirc}^{*}$ denote the set of finite sequences of variables $\cup_{i < \omega} \mathbb{V}^i$ and the set of finite sequences of terms $\cup_{i < \omega} \Lambda_{\bigcirc}^i$ respectively. In the term $({\bf let \:} \bigcirc \overrightarrow{x} = \overrightarrow{M} {\: \bf in \:} N)$, the sequence of variables $\overrightarrow{x}$ and the sequence of terms $\overrightarrow{M}$ should have the same length. Otherwise, such a term is not well-formed.

As we discuss below, the terms of the form ${\bf let \:} \bigcirc \overrightarrow{x} = \overrightarrow{M} {\: \bf in \:} N$ correspond to the special local binding.

\begin{defin} The set of types:

Let $\mathbb{T} = \{ p_0, p_1, \dots \}$ be the set of atomic types, the set $\mathbb{T}_{\bigcirc}$ of types is generated by the grammar:
\begin{center}
  $\mathbb{T}_{\bigcirc} ::= \mathbb{T} \: | \: (\mathbb{T}_{\bigcirc} \to \mathbb{T}_{\bigcirc}) \: | \: (\mathbb{T}_{\bigcirc} \times \mathbb{T}_{\bigcirc}) \: | \: (\bigcirc \mathbb{T}_{\bigcirc})$
\end{center}
\end{defin}

A context has the standard definition \cite{Neder} as a sequence of type declarations
$\Gamma = \{ x_0 : \varphi_1, \dots, x_n : \varphi_{n - 1} \}$. Here $x_i$ is a variable and $\varphi_i$ is a type for each $i < n < \omega$.

\begin{defin} The modal lambda calculus $\lambda_{{\bf IEL}^{-}}$:

  \begin{center}
  \begin{prooftree}
  \AxiomC{$ $}
  \RightLabel{\scriptsize{ax}}
  \UnaryInfC{$\Gamma, x : \varphi \vdash x : \varphi$}
  \end{prooftree}
  \end{center}

  \begin{minipage}{0.45\textwidth}
    \begin{prooftree}
    \AxiomC{$\Gamma, x : \varphi \vdash M : \psi$}
    \RightLabel{$\rightarrow_i$}
    \UnaryInfC{$\Gamma \vdash \lambda x. M : \varphi \to \psi$}
    \end{prooftree}

    \begin{prooftree}
    \AxiomC{ $\Gamma \vdash M : \varphi$ }
    \AxiomC{ $\Gamma \vdash N : \psi$ }
    \RightLabel{$\times_i$}
    \BinaryInfC{$\Gamma \vdash \langle M, N \rangle : \varphi \times \psi$}
    \end{prooftree}

    \begin{prooftree}
      \AxiomC{$\Gamma \vdash M : \varphi$}
      \RightLabel{$\bigcirc_I$}
      \UnaryInfC{$\Gamma \vdash {\bf pure \: } \: M : \bigcirc \varphi$}
    \end{prooftree}
\end{minipage}%
\hfill
\begin{minipage}{0.45\textwidth}
\begin{tabular}{p{\textwidth}}
  \begin{prooftree}
  \AxiomC{$\Gamma \vdash M : \varphi \to \psi$}
  \AxiomC{$\Gamma \vdash N : \varphi$}
  \RightLabel{$\rightarrow_e$}
  \BinaryInfC{$\Gamma \vdash M N : \psi$}
  \end{prooftree}

  \begin{prooftree}
  \AxiomC{ $\Gamma \vdash M : \varphi_1 \times \varphi_2$ }
  \RightLabel{$\times_e$, $i = 1, 2$}
  \UnaryInfC{$\Gamma \vdash \pi_i M : \varphi_i$}
  \end{prooftree}

  \begin{prooftree}
    \AxiomC{$\Gamma \vdash \overrightarrow{M} : \bigcirc \overrightarrow{\varphi}$}
    \AxiomC{$\overrightarrow{x} : \overrightarrow{A} \vdash N : \psi$}
    \RightLabel{$\text{let}_{\bigcirc}$}
    \BinaryInfC{$\Gamma \vdash {\bf let \:} \bigcirc \overrightarrow{x} = \overrightarrow{M} {\: \bf in \: } N : \bigcirc \psi$}
  \end{prooftree}
\end{tabular}
\end{minipage}%
\end{defin}

$\Gamma \vdash \overrightarrow{M} : \bigcirc \overrightarrow{\varphi}$ is a short form for the sequence $\Gamma \vdash M_1 : \bigcirc \varphi_1,\dots,\Gamma \vdash M_n : \bigcirc \varphi_n$ and $\overrightarrow{x} : \overrightarrow{\varphi} \vdash N : \psi$ is a short form for $x_1 : \varphi_1, \dots, x_n : \varphi_n \vdash N : B$.
We use this short form instead of ${\bf let \:} \bigcirc x_1,\dots,x_n = M_1,\dots,M_n {\: \bf in \:} N$. The $\bigcirc_I$-typing rule is the same as $\bigcirc$-introduction in monadic metalanguage \cite{Lax}. $\bigcirc_I$ injects an object of type $A$ into $\bigcirc$. According to this rule, the type constructor ${\bf pure}$ reflects the method \verb"pure" in the \verb"Applicative" class.

The rule $\text{let}_{\bigcirc}$ is similar to the $\bigcirc$-rule in typed lambda calculus for intuitionistic normal modal logic
${\bf IK}$, see \cite{ModalK1}. Informally, one may read ${\bf let \:} \bigcirc \overrightarrow{x} =
\overrightarrow{M} {\: \bf in \: } N$ as a simultaneous local binding in $N$, where each free variable of a term $N$ should be
binded with term of modalised type from $\overrightarrow{M}$. In other words, we modalise all free variables of a term $N$ and
``substitute'' them to the terms belonging to the sequence $\overrightarrow{M}$.

Our calculus extends the typed lambda calculus for {\bf IK} with $\bigcirc_I$-rule with the co-reflection rule allowing one to modalise any type of an arbitrary context.

Here are some examples:

\begin{prooftree}
\AxiomC{$x : \varphi \vdash x : \varphi$}
\RightLabel{$\bigcirc_I$}
\UnaryInfC{$x : \varphi \vdash {\bf pure \:} x : \bigcirc \varphi$}
\RightLabel{$\to_I$}
\UnaryInfC{$\vdash (\lambda x. {\bf pure \: } x) : \varphi \to \bigcirc \varphi$}
\end{prooftree}

\begin{small}
\begin{prooftree}
\AxiomC{$f : \bigcirc (\varphi \to \psi) \vdash f : \bigcirc (\varphi \to \psi)$}
\AxiomC{$x : \bigcirc \varphi \vdash x : \bigcirc \varphi$}
\AxiomC{$g : \varphi \to \psi \vdash g : \varphi \to \psi$}
\AxiomC{$y : \varphi \vdash \varphi : \psi$}
\RightLabel{$\to_e$}
\BinaryInfC{$g : \varphi \to \psi, y : \varphi \vdash g y : \psi$}
\RightLabel{$\text{let}_{\bigcirc}$}
\TrinaryInfC{$f : \bigcirc (\varphi \to \psi), x : \bigcirc \varphi \vdash {\bf let \:} \bigcirc g, y  = f, x {\: \bf in \:} g y : \bigcirc \psi$}
\RightLabel{$\to_I$}
\UnaryInfC{$f : \bigcirc (\varphi \to \psi) \vdash \lambda x. {\bf let \:} \bigcirc g, y = f, x {\: \bf in \:} g y : \bigcirc \varphi \to \bigcirc \psi$}
\RightLabel{$\to_I$}
\UnaryInfC{$\vdash \lambda f. \lambda x. {\bf let \:} \bigcirc g, y = f, x {\: \bf in \:} g y : \bigcirc (\varphi \to \psi) \to \bigcirc \varphi \to \bigcirc \psi$}
\end{prooftree}
\end{small}

Here we provided the derivations for modal axioms of ${\bf IEL}^{-}$. In fact, we proved Lemma 2 using proof-assignment.

\vspace{\baselineskip}

Now we define free variables and substitutions:

\begin{defin} The set $FV(M)$ of free variables for a term $M$:

\begin{enumerate}
\item $FV(x) = \{ x \}$.
\item $FV(\lambda x. M) = FV(M) \setminus \{ x\}$.
\item $FV(M N) = FV(M) \cup FV(N)$.
\item $FV(\langle M,N \rangle) = FV(M) \cup FV(N)$.
\item $FV(\pi_i M) = FV(M)$, $i = 1, 2$.
\item $FV(\text{\bf pure } M) = FV(M)$.
\item $FV({\bf let \:} \bigcirc \: \overrightarrow{x} = \overrightarrow{M} {\: \bf in \:} N) = \cup_{i = 1}^n FV(M), \text{where $n = |\overrightarrow{M}|$}$.
\end{enumerate}
\end{defin}

\begin{defin} Substitution:

\begin{enumerate}
\item $x [x := N] = N$, $x [y := N] = x$.
\item $(M N) [x := N] = M[x := N] N [x := N]$.
\item $(\lambda x. M) [y := N] = \lambda x. M [y := N]$, $y \in FV(M)$.
\item $(M, N)[x := P] = (M[x := P], N [x := P])$.
\item $(\pi_i M) [x := P] = \pi_i (M[x := P])$, $i = 1, 2$.
\item $({\bf pure \: } M) [x := P] = {\bf pure \: } (M [x := P])$.
\item $({\bf let \:} \bigcirc \overrightarrow{x} = \overrightarrow{M} {\: \bf in \:} N) [y := P] = {\bf let \:} \bigcirc \overrightarrow{x} = (\overrightarrow{M} [y := P]) {\: \bf in \:} N$.
\end{enumerate}
\end{defin}

Substitutions and free variables for terms of the kind ${\bf let \:} \bigcirc \overrightarrow{x} = \overrightarrow{M} {\: \bf in \:} N$ are defined similarly to \cite{ModalK1}.
That is, we do not take into account free variables of $N$ because those variables occur in the list $\overrightarrow{x}$ and are eliminated by the assignment $\overrightarrow{x} = \overrightarrow{M}$.

The reduction rules are the following ones:

  \begin{defin} $\beta$-reduction rules for $\lambda_{{\bf IEL}^{-}}$.

\begin{enumerate}
  \item $(\lambda x. M) N \rightarrow_{\beta} M [x := N]$.
  \item $\pi_1 \langle M, N \rangle \rightarrow_{\beta} M$.
  \item $\pi_2 \langle M, N \rangle \rightarrow_{\beta} N$.
  \item ${\bf let \:} \bigcirc \overrightarrow{x}, y, \overrightarrow{z} = \overrightarrow{M}, {\bf let \:} \bigcirc \overrightarrow{w} = \overrightarrow{N} {\: \bf in \: } Q, \overrightarrow{P} {\: in \:} R \rightarrow_{\beta} \\
  {\bf let \:} \bigcirc \overrightarrow{x}, \overrightarrow{w}, \overrightarrow{z} = \overrightarrow{M}, \overrightarrow{N}, \overrightarrow{P} {\: \bf in \: } R [y := Q]$.
  \item ${\bf let \:} \bigcirc \overrightarrow{x} = {\bf pure \:} \overrightarrow{M} {\: \bf in \:} N \rightarrow_{\beta} {\bf pure \:} N [\overrightarrow{x} := \overrightarrow{M}]$.
  \item ${\bf let \:} \bigcirc \underline{\quad} = \underline{\quad} {\: \bf in \:} M \rightarrow_{\beta} {\bf pure \:} M$, where \underline{\quad} is an empty sequence of terms.
\end{enumerate}
\end{defin}

If $M$ reduces to $N$ by one of these rules, then we write $M \rightarrow_{r} N$. A multistep reduction $\twoheadrightarrow_{r}$ is a reflexive transitive closure of $\rightarrow_{r}$. $=_r$ is a symmetric closure of $\twoheadrightarrow_{r}$. Now we formulate the standard lemmas.

\begin{prop} The generation lemma for $\bigcirc_I$. \label{Gen}

  Let $\Gamma \vdash {\bf pure \:} M : \bigcirc \varphi$, then $\Gamma \vdash M : \varphi$.
\end{prop}

\begin{proof}
  Straightforwardly.
\end{proof}

\begin{lemma} Basic lemmas. \label{Struc}

\begin{enumerate}
  \item If $\Gamma \vdash M : \varphi$ and $\Gamma \subseteq \Delta$, then $\Delta \vdash M : \varphi$.
  \item If $\Gamma \vdash M : \varphi$, then $\Delta \vdash M : \varphi$, where $\Delta = \{ x : \psi \: | \: (x : \psi) \in \Gamma \:\: \& \:\: x \in FV(M) \}$.
  \item If $\Gamma, x : \varphi \vdash M : \phi$ and $\Gamma \vdash N : \varphi$, then $\Gamma \vdash M [x := N] : \psi$.
\end{enumerate}
\end{lemma}

\begin{proof}
  $ $

The items 1-2 are proved by induction on $\Gamma \vdash M : \varphi$ using \ref{Gen}. The third item is shown by induction on the derivation of $\Gamma \vdash N : \psi$.

\end{proof}

\begin{theorem} Subject reduction.

  If $\Gamma \vdash M : \varphi$ and $M \twoheadrightarrow_r N$, then $\Gamma \vdash N : \varphi$.

\end{theorem}

\begin{proof}

By induction on the derivation $\Gamma \vdash M : \varphi$ and on the generation of $\rightarrow_{\beta}$. The general statement follows from transitivity of $\twoheadrightarrow_{\beta}$, Proposition~\ref{Gen}, and Lemma~\ref{Struc}.
\end{proof}

Now we discuss the relation between the monadic metalanguage and $\lambda_{{\bf IEL}^{-}}$. The monadic metalanguage is the modal lambda-calculus based on the categorical semantics of computation proposed by Moggi \cite{moggi1991notions}. As we mentioned above, the monadic metalanguage might be considered as the type-theoretical representation of computation with an abstract data type of action. In fact, the monadic metalanguage is a type-theoretical formulation of Haskell monads. We show that $\lambda_{{\bf IEL}^{-}}$ is sound with respect to the monadic metalanguage.

\begin{defin} The monadic metalanguage
  $ $

The monadic metalanguage extends the simply-typed lambda calculus with the additional typing rules:

  \begin{minipage}{0.45\textwidth}
    \begin{prooftree}
      \AxiomC{$\Gamma \vdash M : \varphi$}
      \RightLabel{$\nabla_I$}
      \UnaryInfC{$\Gamma \vdash {\bf val \: } M : \nabla \varphi$}
    \end{prooftree}
  \end{minipage}%
  \hfill
  \begin{minipage}{0.45\textwidth}
  \begin{tabular}{p{\textwidth}}
  \begin{prooftree}
    \AxiomC{$\Gamma \vdash M : \nabla \varphi$}
    \AxiomC{$\Gamma, x : \varphi \vdash N : \nabla \psi$}
    \RightLabel{$\text{let}_{\nabla}$}
    \BinaryInfC{$\Gamma \vdash {\bf let \: val \:} x = M {\: \bf in \: } N : \nabla \psi$}
  \end{prooftree}
  \end{tabular}
  \end{minipage}
\end{defin}

The reduction rules are the following ones (in addition to the standard rule for abstraction and application):

\begin{enumerate}
  \item ${\bf let \: val \:} x = {\bf val \: } M {\bf \: in \:} N \rightarrow_{\beta} N [x := M]$
  \item ${\bf let \: val \:} x = ({\bf let \: val \:} y = N {\bf \: in \:} P) {\bf \: in \:} M \rightarrow_{\beta}
  {\bf let \: val \:} y = N {\bf \: in \:} ({\bf let \: val \:} x = P {\bf \: in \:} M)$
  \item ${\bf let \: val \:} x = M {\bf \: in \:} x \rightarrow_{\eta} M$
\end{enumerate}

Let us define the translation $\ulcorner . \urcorner$ from $\lambda_{{\bf IEL}^{-}}$ to the monadic metalanguage:

\begin{enumerate}
  \item $\ulcorner p_i \urcorner = p_i$, where $p_i$ is atomic
  \item $\ulcorner \varphi \to \psi \urcorner = \ulcorner \varphi \urcorner \to \ulcorner \psi \urcorner$
  \item $\ulcorner \bigcirc \varphi \urcorner = \nabla \ulcorner \varphi \urcorner$
\end{enumerate}

\begin{enumerate}
  \item $\ulcorner x \urcorner = x$, $x$ is a variable
  \item $\ulcorner \lambda x. M \urcorner = \lambda x. \ulcorner M \urcorner$
  \item $\ulcorner M \: N \urcorner = \ulcorner M \urcorner \ulcorner N \urcorner$
  \item $\ulcorner {\bf pure \:} M \urcorner = {\bf val \:} \ulcorner M \urcorner$
  \item $\ulcorner {\bf let \:} \bigcirc \overrightarrow{x} = \overrightarrow{M} {\bf \: in \:} N \urcorner = {\bf let \: val \:} \overrightarrow{x} = \ulcorner \overrightarrow{M}
  \urcorner {\bf \: in \:} {\bf val} \ulcorner N \urcorner$
\end{enumerate}
where ${\bf let \: val \:} \overrightarrow{x} = \ulcorner \overrightarrow{M} \urcorner {\bf \: in \:} N$ denotes ${\bf let \: val \:} x_1 = \ulcorner M_1 \urcorner {\bf \: in \:} ( \dots {\bf \: in \:} ({\bf let \: val \:} x_n = \ulcorner M_n \urcorner {\bf \: in \:} {\bf val \:} N) \dots)$

If $\Gamma = \{ x_1 : \varphi_1, \dots, x_n : \varphi_n \}$ is a context, then $\ulcorner \Gamma \urcorner = \{ x_1 : \ulcorner \varphi_1 \urcorner, \dots, x_n : \ulcorner \varphi_n \urcorner \}$. Let us denote $\vdash_{\lambda_{{\bf IEL}^{-}}}$ as the derivability relation in $\lambda_{{\bf IEL}^{-}}$ in order to distinguish the $\lambda_{{\bf IEL}^{-}}$ derivability from the monadic metalanguage one.

\begin{lemma}
  $ $

  If $\Gamma \vdash_{\lambda_{{\bf IEL}^{-}}} M : A$, then $\ulcorner \Gamma \urcorner \vdash \ulcorner M \urcorner : \ulcorner A \urcorner$ in the monadic metalanguage.
\end{lemma}

\begin{proof} By induction on $\Gamma \vdash_{\lambda_{{\bf IEL}^{-}}} M : A$. One may prove the cases of $\Box_I$ and $\text{let}_{\Box}$ as follows:

  \begin{prooftree}
    \AxiomC{$\ulcorner \Gamma \urcorner \vdash \ulcorner M \urcorner : \ulcorner A \urcorner$}
    \UnaryInfC{$\ulcorner \Gamma \urcorner \vdash {\bf val \:} \ulcorner M \urcorner : \nabla \ulcorner A \urcorner$}
  \end{prooftree}

  \begin{prooftree}
    \AxiomC{$\ulcorner \Gamma \urcorner \vdash \ulcorner \overrightarrow{M} \urcorner : \nabla \ulcorner \overrightarrow{A} \urcorner$}
    \AxiomC{$\overrightarrow{x} : \ulcorner \overrightarrow{A} \urcorner \vdash \ulcorner N \urcorner : \ulcorner B \urcorner$}
    \UnaryInfC{$\overrightarrow{x} : \ulcorner \overrightarrow{A} \urcorner \vdash {\bf val \:} \ulcorner N \urcorner : \nabla \ulcorner B \urcorner$}
    \BinaryInfC{$\ulcorner \Gamma \urcorner \vdash {\bf let \: val \:} \overrightarrow{x} = \ulcorner \overrightarrow{M} \urcorner {\: \bf in \:} {\bf val \:} \ulcorner N \urcorner : \nabla \ulcorner B \urcorner$}
  \end{prooftree}
\end{proof}

Now one may formulate the following lemma:

\begin{lemma}
  $ $

  \begin{enumerate}
    \item $\ulcorner M [x := N] \urcorner = \ulcorner M \urcorner [x := \ulcorner N \urcorner]$
    \item $M \twoheadrightarrow_{r} N \Rightarrow \ulcorner M \urcorner \twoheadrightarrow_{\beta} \ulcorner N \urcorner$
  \end{enumerate}
\end{lemma}

\begin{proof}
$ $

\begin{enumerate}
  \item Induction on the structure of $M$.
  \item By the induction on $\to_r$:
  \begin{enumerate}
  \item For simplicity, we consider the case with only one variable in ${\bf let \:} \bigcirc$ local binding, that can be easily extended to an arbitrary number of variables in local binding:

    $\begin{array}{lll}
    &\ulcorner {\bf let \:} \bigcirc x = ({\bf let \:} \bigcirc \overrightarrow{y} = \overrightarrow{N} {\: \bf in \:} P) {\: \bf in \:} M \urcorner = & \\
    &\quad {\bf let \: val \:} x = ({\bf let \: val \:} \overrightarrow{y} = \ulcorner \overrightarrow{N} \urcorner {\: \bf in \:} {\bf val \:} \ulcorner P \urcorner) {\: \bf in \:} {\bf val \:} \ulcorner M \urcorner \rightarrow_{\beta} & \\
    &\quad {\bf let \: val \:} \overrightarrow{y} = \ulcorner \overrightarrow{N} \urcorner {\: \bf in \:} ({\bf let \: val \:} x = \ulcorner P \urcorner {\: \bf in \:} {\bf val \:} \ulcorner M \urcorner) \rightarrow_{\beta} & \\
    &\quad {\bf let \: val \:} \overrightarrow{y} = \ulcorner \overrightarrow{N} \urcorner {\: \bf in \:} {\bf val \:} \ulcorner M \urcorner [x := \ulcorner P \urcorner] = \ulcorner {\bf let \:} \bigcirc \overrightarrow{y} = \overrightarrow{N} {\: \bf in \:} M [x := P] \urcorner&
    \end{array}$

    \item $ $

    $\begin{array}{lll}
    &\ulcorner {\bf let \:} \bigcirc \overrightarrow{x} = {\bf pure \:} \overrightarrow{N} {\: \bf in \:} M \urcorner =  {\bf let \: val \:} \overrightarrow{x} = {\bf val \:} \ulcorner \overrightarrow{N} \urcorner {\: \bf in \:} {\bf val} \ulcorner M \urcorner \rightarrow_{\beta} & \\
    &\quad {\bf val \:} \ulcorner M \urcorner [\overrightarrow{x} := \ulcorner \overrightarrow{N} \urcorner] = \ulcorner {\bf pure \:} M [\overrightarrow{x} := \overrightarrow{N}] \urcorner&
    \end{array}$

    \item $\ulcorner {\bf let \:} \bigcirc x = M {\: \bf in \:} x \urcorner = {\bf let \: val \:} x = \ulcorner M \urcorner {\: \bf in \:} {\bf val \:} x \rightarrow_{\eta} \ulcorner M \urcorner$
  \end{enumerate}
\end{enumerate}
\end{proof}

\begin{theorem} \label{MonadSound}
  $ $

${\bf IEL}^{-}$ is sound with respect to the monadic metalanguage.

\end{theorem}

\begin{proof}
  Follows from the lemmas above.
\end{proof}

\begin{theorem}
  $ $

  $\twoheadrightarrow_{\beta}$ is strongly normalising.
\end{theorem}

\begin{proof}
  Follows from Theorem~\ref{MonadSound} above, so far as reduction in the monadic metalanguage is strongly normalising \cite{Paiva2} and $\lambda_{{\bf IEL}^{-}}$ is sound with respect to the monadic metalanguage.
\end{proof}

\begin{theorem}
  $ $

  $\twoheadrightarrow_r$ is confluent.
\end{theorem}

\begin{proof}
  $ $

By Newman's lemma \cite{newman1942theories}, if a relation is strongly normalising and locally confluent, then this relation is confluent. It is sufficient to show that a multistep reduction is locally confluent.

\begin{lemma} If $M \rightarrow_{r} N$ and $M \rightarrow_{r} Q$, then there exists some term $P$,
such that $N \twoheadrightarrow_{r} P$ and $Q \twoheadrightarrow_{r} P$.

\end{lemma}

\begin{proof}

Let us consider the following critical pairs and show that they are joinable:

\begin{enumerate}
\item $ $ \\

\xymatrix{
{\bf let \:} \bigcirc x = ({\bf let \:} \bigcirc \overrightarrow{y} = {\bf pure \:} \overrightarrow{N} {\: \bf in \:} P) {\: \bf in \:} M \ar[d]_{\beta } \ar[dr]^{\beta} \\
{\bf let \:} \bigcirc \overrightarrow{y} = {\bf pure \:} \overrightarrow{N} {\: \bf in \:} M [x := P] & {\bf let \:} \bigcirc x = {\bf pure \:} P [\overrightarrow{y} := \overrightarrow{N}] {\: \bf in \:} M
}

\vspace{\baselineskip}

$\begin{array}{lll}
&{\bf let \:} \bigcirc \overrightarrow{y} = {\bf pure \:} \overrightarrow{N} {\: \bf in \:} M [x := P] \rightarrow_{\beta}& \\
&\quad\quad\quad\quad\quad\quad\quad {\bf pure \:} M [x := P] [\overrightarrow{y} := \overrightarrow{N}]& \\
&{\bf let \:} \bigcirc x = {\bf pure \:} P [\overrightarrow{y} := \overrightarrow{N}] {\: \bf in \:} M \rightarrow_{\beta}& \\
&\quad\quad\quad\quad\quad\quad\quad {\bf pure \:} M [x := P[\overrightarrow{y} := \overrightarrow{N}]] \equiv & \\
&\text{Since $x \notin \overrightarrow{y}$}& \\
&\quad\quad\quad\quad\quad\quad\quad {\bf pure \:} M [x := P] [\overrightarrow{y} := \overrightarrow{N}] & \\
\end{array}$

\item $ $ \\

\xymatrix{
{\bf let \:} \bigcirc x = ({\bf let \:} \bigcirc \underline{\quad} = \underline{\quad} {\: \bf in \:} N) {\: \bf in \:} M \ar[d]^{\beta} \ar[dr]^{\beta} \\
{\bf let \:} \bigcirc \underline{\quad} = \underline{\quad} {\: \bf in \:} M [x := N] & {\bf let \:} \bigcirc x = {\bf pure \:} N {\: \bf in \:} M
}

$\begin{array}{lll}
&{\bf let \:} \bigcirc \underline{\quad} = \underline{\quad} {\: \bf in \:} M [x := N] \rightarrow_{\beta } {\bf let \:} \bigcirc (M [x := N])& \\
&{\bf let \:} \bigcirc x = {\bf pure \:} N {\: \bf in \:} M \rightarrow_{\beta} {\bf pure \:} (M [x := N])&
\end{array}$
\end{enumerate}
\end{proof}

One may consider four critical pairs analysed in the confluence proof for the lambda-calculus based on the intuitionistic normal modal logic {\bf IK} \cite{ModalK1}. Those pairs are joinable in our calculus as well.

\end{proof}

\subsection{Categorical semantics}

In this subsection, we provide categorical semantics for the modal lambda calculus proposed above.  We consider the co-reflection principle coalgebraically. We recall the required notions first. See the underlying categorical definitions here \cite{goldblatt2014topoi}. We piggyback the construction used in the proof of the completeness for the simply-typed lambda-calculus, see \cite{abramsky2010introduction} to have comprehensive details.

\begin{defin} A category $\mathcal{C}$ is called cartesian closed if it has products $A \times B$, exponentials $B^A$ and the terminal object $\mathds{1}$ satisfying the universal product and exponentiation properties.
\end{defin}

Following to Bellin et. al. \cite{bellin2001extended} and Kakutani \cite{ModalK1} \cite{kakutani2016calculi}, we interpret a modal operator as a monoidal endofunctor on a cartesian closed category. A monoidal endofunctor is a functor that respects monoidal structure, products and the terminal object in our case. Here we refer to the work by Eilenberg and Kelly for precise details \cite{10.1007/978-3-642-99902-4_22}. We define a monoidal endofunctor on a cartesian closed category as an underlying notion.

\begin{defin} Let $\mathcal{C}$ be a cartesian closed category and ${\bf F} : \mathcal{C} \to \mathcal{C}$ an endofunctor, ${\bf F}$ is called monoidal if there exists a natural transformation $m$ consisting of components $m_{A, B} : {\bf F} A \times {\bf F} B \to {\bf F}(A \times B)$ and a natural transformation $u : \mathds{1} \to {\bf F} \mathds{1}$ such that the well-known diagrams commute (MacLane pentagon and triangle identity).
\end{defin}

The abstract definition of a coalgebra is the following one \footnote{A coalgebraic technique os widely used in logic and computer science as well, see \cite{venema2006algebras}.}:

\begin{defin}
  Let $\mathcal{C}$ be a category and ${\bf F} : \mathcal{C} \to \mathcal{C}$ an endofunctor. If $A \in \operatorname{Ob}(\mathcal{C})$, then an ${\bf F}$-coalgebra is a pair $\langle A, \alpha \rangle$, where $\alpha \in \operatorname{Hom}_{\mathcal{C}}(A, {\bf F} A)$.
  An ${\bf F}$-coalgebra homomorpism from $\langle A, \alpha \rangle$ to $\langle A, \beta \rangle$ is a map $f \in \operatorname{Hom}_{\mathcal{C}}(A, B)$ such that the following square commutes:

  \xymatrix{
  &&&&&& A \ar[r]^{\alpha} \ar[d]_{f} & {\bf F} A \ar[d]^{{\bf F} f}\\
  &&&&&& B \ar[r]_{\beta} & {\bf F} B
  }
\end{defin}

Given a natural transformation $\alpha : \operatorname{Id}_{\mathcal{C}} \to {\bf F}$, one may associate an ${\bf F}$-coalgebra $\langle A, \alpha_A \rangle$ for each $A \in \operatorname{Ob}(\mathcal{C})$. Homomorphisms of such coalgebras are defined by naturality.

\begin{defin}
  Let $\mathcal{C}$ be a cartesian closed category, ${\bf F} : \mathcal{C} \to \mathcal{C}$ a monoidal functor on $\mathcal{C}$, and $\alpha : \operatorname{Id}_{\mathcal{C}} \to {\bf F}$ a natural transformation. An ${\bf IEL}^{-}$-category is a pair $\langle \mathcal{C}, {\bf F}, \alpha \rangle$ such that the following coherence conditions hold:

  \begin{enumerate}
    \item $u = \alpha_{\mathds{1}}$, where $\alpha_{\mathds{1}}$
    \item $m_{A,B} \circ (\alpha_A \times \alpha_B) = \alpha_{A \times B}$, i.e. the following diagram commutes:

    \xymatrix
    {
    &&&&& A \times B \ar[rr]^{\alpha_A \times \alpha_B} \ar[drr]_{\alpha_{A \times B}} && {\bf F}A \times {\bf F}B \ar[d]^{m_{A,B}} \\
    &&&&&&& {\bf F}(A \times B)
    }
  \end{enumerate}
\end{defin}

The following construction describes the standard construction of typed lambda-calculus semantics \cite{abramsky2010introduction} \cite{lambek1988introduction}. First of all, let us define semantic brackets $[\![.]\!]$, a semantic translation from $\lambda_{\text{IEL}^{-}}$ to the ${\bf IEL}^{-}$-category $\langle \mathcal{C}, {\bf F}, \alpha \rangle$. Suppose one has an assignment $\hat{\cdot}$ that maps every primitive type to some object of $\mathcal{C}$. Such semantic brackets $[\![.]\!]$ have the following inductive definition:

\begin{enumerate}
\item $[\![p_i]\!] := \hat{p_i}$
\item $[\![\varphi \to \psi]\!] := [\![\varphi]\!]^{[\![\psi]\!]}$
\item $[\![\varphi \times \psi]\!] := [\![\varphi]\!] \times [\![\psi]\!]$
\item $[\![\bigcirc \varphi]\!] = {\bf F} [\![\varphi]\!]$
\end{enumerate}

We extend this interpretaion for contexts by induction too:
  \begin{enumerate}
    \item $[\![ \quad ]\!] = \mathds{1}$, where $\mathds{1}$ is a terminal object of a given CCC
    \item $[\![\Gamma, x : \varphi]\!] = [\![\Gamma]\!] \times [\![\varphi]\!]$
  \end{enumerate}

The typing rules are interpreted as follows. We understand typing assignments $\Gamma \vdash M : A$ as arrows of the form $[\![\Gamma \vdash M : \varphi]\!] = [\![M]\!] : [\![\Gamma]\!] \to [\![\varphi]\!]$.

\begin{minipage}{0.5\textwidth}
  \begin{flushleft}

    \begin{prooftree}
    \AxiomC{$ $}
    \UnaryInfC{$\pi_2 : [\![\Gamma]\!] \times [\![\varphi]\!] \rightarrow [\![\varphi]\!]$}
    \end{prooftree}

    \begin{prooftree}
    \AxiomC{$[\![M]\!] : [\![\Gamma]\!] \times [\![\varphi]\!] \rightarrow [\![\psi]\!]$}
    \UnaryInfC{$\Lambda([\![M]\!]) : [\![\Gamma]\!] \rightarrow[\![\psi]\!]^{[\![\varphi]\!]}$}
    \end{prooftree}

    \begin{prooftree}
    \AxiomC{$[\![M]\!] : [\![\Gamma]\!] \rightarrow [\![\varphi]\!]$}
    \AxiomC{$[\![N]\!] : [\![\Gamma]\!] \rightarrow [\![\psi]\!]$}
    \BinaryInfC{$\langle [\![M]\!], [\![N]\!] \rangle : [\![\Gamma]\!] \rightarrow [\![\varphi]\!] \times [\![\psi]\!]$}
    \end{prooftree}
  \end{flushleft}
\end{minipage}
\begin{minipage}{0.5\textwidth}
  \begin{flushright}

    \begin{prooftree}
    \AxiomC{$[\![M]\!] : [\![\Gamma]\!] \rightarrow [\![\varphi]\!]$}
    \UnaryInfC{$[\![M]\!] \circ \eta_{[\![\varphi]\!]} : [\![\Gamma]\!] \rightarrow {\bf F} [\![\varphi]\!]$}
    \end{prooftree}

    \begin{prooftree}
    \AxiomC{$[\![M]\!] : [\![\Gamma]\!] \rightarrow [\![\psi]\!]^{[\![\varphi]\!]}$}
    \AxiomC{$[\![N]\!] : [\![\Gamma]\!] \rightarrow [\![\varphi]\!]$}
    \BinaryInfC{$\langle [\![M]\!], [\![N]\!] \rangle \circ \epsilon_{[\![\varphi]\!],[\![B]\!]} : [\![\Gamma]\!] \rightarrow [\![\psi]\!] $}
    \end{prooftree}

    \begin{prooftree}
    \AxiomC{$[\![M]\!] : [\![\Gamma]\!] \rightarrow [\![\varphi_1]\!] \times [\![\varphi_2]\!]$}
    \RightLabel{$i = 1,2$}
    \UnaryInfC{$[\![M]\!] \circ \pi_i : [\![\Gamma]\!] \rightarrow [\![\varphi_i]\!]$}
    \end{prooftree}
  \end{flushright}
\end{minipage}

\begin{prooftree}
    \AxiomC{$\langle [\![M_1]\!],\dots, [\![M_n]\!] \rangle : [\![\Gamma]\!] \rightarrow \prod \limits_{i=1}^n {\bf F} [\![\varphi_i]\!]$}
    \AxiomC{$[\![N]\!] : \prod \limits_{i=1}^n [\![\varphi_i]\!] \rightarrow [\![\psi]\!]$}
    \BinaryInfC{${\bf F} ([\![N]\!]) \circ m_{[\![\varphi_1]\!], \dots, [\![\varphi_n]\!]} \circ \langle [\![M_1]\!],\dots, [\![M_n]\!] \rangle : [\![\Gamma]\!] \rightarrow {\bf F} [\![\psi]\!]$}
\end{prooftree}

The ${\bf let} \bigcirc$-rule has the intepretation similar to $\Box$-rule in term calculus for intutionistic {\bf K} \cite{bellin2001extended}. The semantic brackets respect all substitutions and reductions due to this lemma:

\begin{lemma}
$ $

\begin{enumerate}
\item $[\![M [x_1 := M_1,\dots, x_n := M_n]]\!] = [\![M]\!] \circ \langle [\![M_1]\!], \dots, [\![M_n]\!] \rangle$.
\item If $\Gamma \vdash M : A$ and $M \to_r N$, then $[\![\Gamma \vdash M : A]\!] = [\![\Gamma \vdash N : A]\!]$.
\end{enumerate}

\end{lemma}

\begin{proof}
$ $

  \begin{enumerate}
    \item By simple induction on $M$. Let us check only the modal cases.

    $\begin{array}{lll}
    & [\![\Gamma \vdash ({\bf pure \:} M) [ \vec{x} := \vec{M}] : \bigcirc \varphi]\!] = [\![ \Gamma \vdash {\bf pure \:} (M [ \vec{x} := \vec{M}]) : \bigcirc \varphi]\!] = & \\
    &\eta_{[\![A]\!]} \circ [\![(M [ \vec{x} := \vec{M}])]\!] = \alpha_{[\![A]\!]} \circ ([\![M]\!] \circ \langle [\![M_1]\!], \dots, [\![M_n]\!] \rangle) = & \\
    &(\alpha_{[\![A]\!]} \circ [\![M]\!]) \circ \langle [\![M_1]\!], \dots, [\![M_n]\!] \rangle = [\![ \Gamma \vdash {\bf pure \:} M : \bigcirc \varphi]\!] \circ \langle [\![M_1]\!], \dots, [\![M_n]\!] \rangle& \\
    \end{array}$

    \vspace{\baselineskip}

    $\begin{array}{lll}
    &[\![\Gamma \vdash ({\bf let} \bigcirc \: \vec{x} = \vec{M} {\: \bf in \:} N) [\vec{y} := \vec{P}] : \bigcirc \psi]\!] = [\![\Gamma \vdash {\bf let} \bigcirc \vec{x} = (\vec{M} [\vec{y} := \vec{P}]) {\: \bf in \:} N : \bigcirc \psi]\!] =& \\
    &{\bf F} ([\![N]\!]) \circ m_{[\![\varphi_1]\!], \dots, [\![\varphi_n]\!]} \circ [\![\Gamma \vdash (\vec{M} [\vec{y} := \vec{P}]) : \bigcirc \vec{\varphi}]\!] =& \\
    &{\bf F} ([\![N]\!]) \circ m_{[\![\varphi_1]\!], \dots, [\![\varphi_n]\!]} \circ [\![\vec{M}]\!] \circ \langle [\![P_1]\!],\dots,[\![P_n]\!]\rangle = & \\
    &[\![\Gamma \vdash {\bf let} \bigcirc \vec{x} = \vec{M} {\: \bf in \:} N : \bigcirc \varphi]\!] \circ \langle [\![P_1]\!],\dots,[\![P_n]\!]\rangle&
    \end{array}$
    \item The cases with $\beta$-reductions for $let_{\bigcirc}$ are shown in \cite{ModalK1}. Those cases are similar to our ones.
    Let us consider the cases with the ${\bf pure}$ terms that immediately follow from the coherence conditions of an ${\bf IEL}^{-}$-category and the previous item of this lemma.
    \begin{enumerate}
      \item $[\![\Gamma \vdash {\bf let} \bigcirc \vec{x} = {\bf pure \:} \vec{M} {\: \bf in \:} N : \bigcirc \psi]\!] = [\![\Gamma \vdash {\bf pure \:} N [\vec{x} := \vec{M}] : \bigcirc \psi]\!]$

      $\begin{array}{lll}
      &[\![\Gamma \vdash {\bf let} \bigcirc \vec{x} = {\bf pure \:} \vec{M} {\: \bf in \:} N : \bigcirc \psi]\!] = & \\
      &{\bf F} ([\![N]\!]) \circ m_{[\![\varphi_1]\!], \dots, [\![\varphi_n]\!]} \circ \langle \alpha_{[\![\varphi_1]\!]} \circ [\![M_1]\!],\dots,\alpha_{[\![\varphi_n]\!]} \circ [\![M_n]\!] \rangle = \\
      &{\bf F} ([\![N]\!]) \circ m_{[\![\varphi_1]\!], \dots, [\![\varphi_n]\!]} \circ (\alpha_{[\![\varphi_1]\!]} \times \dots \times \alpha_{[\![\varphi_n]\!]}) \circ \langle [\![M_1]\!], \dots, [\![M_n]\!] \rangle =& \\
      &{\bf F} ([\![N]\!]) \circ \alpha_{[\![\varphi_1]\!] \times \dots \times [\![\varphi_n]\!]} \circ \langle [\![M_1]\!], \dots, [\![M_n]\!] \rangle = \alpha_{[\![\psi]\!]} \circ [\![N]\!] \circ \langle [\![M_1]\!], \dots, [\![M_n]\!] \rangle =& \\
      &\alpha_{[\![\psi]\!]} \circ [\![\Gamma \vdash N [\vec{x} := \vec{M}] : \psi]\!] = [\![\Gamma \vdash {\bf pure \:} (N [\vec{x} := \vec{M}]) : \bigcirc \psi]\!]&
      \end{array}$
      \item $[\![\vdash {\bf let} \bigcirc \underline{\quad} = \underline {\quad} {\: \bf in \:} M : \bigcirc \varphi ]\!] = [\![\vdash {\bf pure \:} M : \bigcirc \varphi ]\!]$

      $\begin{array}{lll}
      &[\![\vdash \vdash {\bf let} \bigcirc = \underline {\quad} {\: \bf in \:} M : \bigcirc \varphi ]\!] = & \\
      &{\bf F} ([\![M]\!]) \circ u = {\bf F} ([\![M]\!]) \circ \alpha_{\mathds{1}} = \alpha_{[\![A]\!]} \circ [\![M]\!] = [\![\vdash {\bf pure \:} M : \bigcirc \varphi]\!]&
      \end{array}$
    \end{enumerate}
  \end{enumerate}
\end{proof}

The following soundness theorem follows from the lemma above and the whole construction:

\begin{theorem} Soundness

  Let $\Gamma \vdash M : \varphi$ and $M =_r N$, then $[\![\Gamma \vdash M : \varphi]\!] = [\![\Gamma \vdash N : \varphi]\!]$.
\end{theorem}

The completeness theorem is proved with the syntactic model. We will consider the term model for the simply-typed lambda-calculus with $\times$ and $\to$ standardly described in \cite{abramsky2010introduction}. This model is a categorical analogue of the Lindenbaum-Tarski algebra.

Let us define a binary relation on lambda-terms$\sim_{\varphi, \psi} \subseteq (\mathbb{V} \times \Lambda_{\Box})^2$ as:

\begin{center}
  $(x, M) \sim_{\varphi, \psi} (y, N) \Leftrightarrow x : \varphi \vdash M : \psi \:\: \& \:\: y : \varphi \vdash N : \psi \:\: \& \:\: M =_r N [y := x]$
\end{center}

We will denote equivalence class as $[x, M]_{\varphi, \psi} = \{ (y, N) \: | \: (x, M) \sim_{\varphi, \psi} (y, N) \}$ (we will drop indices below). Let us recall the definition of the category $\mathcal{C}(\lambda)$, a model structure for the simply-typed lambda calculus.

The category $\mathcal{C}(\lambda)$ has the class of objects defined as $\operatorname{Ob}_{\mathcal{C}(\lambda)} = \{ \hat{\varphi} \: | \: \varphi \in \mathbb{T} \} \cup \{ \mathds{1} \}$. For $\hat{\varphi},\hat{\psi} \in \operatorname{Ob}_{\mathcal{C}(\lambda)}$,
the set of morphism has the form $\operatorname{Hom}_{\mathcal{C}(\lambda)}(\hat{\varphi},\hat{\psi}) = \{ [x, M] \: | \: x : \varphi \vdash M : \psi \}$. Let $[x, M] \in \operatorname{Hom}_{\mathcal{C}(\lambda)}(\hat{\varphi},\hat{\psi})$ and $[y,N] \in \operatorname{Hom}_{\mathcal{C}(\lambda)}(\hat{\psi},\hat{\theta})$, then $[y,N] \circ [x, M] = [x, N [y := M]]$. Identity morphisms have the form $id_{\hat{\varphi}} = [x,x]$.

The category $\mathcal{C}(\lambda)$ is cartesian closed since $\mathds{1}$ is a terminal object such that $\operatorname{Hom}_{\mathcal{C}(\lambda)}(\mathds{1}, \hat{\varphi}) = \{ [\sqbullet, M] \: | \: \vdash M : \varphi \text{ is provable}\}$; $\widehat{\varphi \times \psi} = \hat{\varphi} \times \hat{\psi}$; and $\widehat{\varphi \to \psi} = \hat{\psi}^{\hat{\varphi}}$. Canonical projections are defined as $[x, \pi_i x] \in \operatorname{Hom}_{\mathcal{C}(\lambda)}(\hat{\varphi_1} \times \hat{\varphi_2},\hat{\varphi_i})$ for $i = 1, 2$. The evaluation arrow is a morphism $ev_{\hat{\varphi},\hat{\psi}} = [x, (\pi_1 x) (\pi_2 x)] \in \operatorname{Hom}_{\mathcal{C}(\lambda)}(\hat{\psi}^{\hat{\varphi}} \times \hat{\varphi}, \hat{\psi})$.

Let us define a map ${\bf F} : \mathcal{C}(\lambda) \to \mathcal{C}(\lambda)$, such that
for all $[x,M] \in \operatorname{Hom}_{\mathcal{C}(\lambda)}(\hat{\varphi},\hat{\psi}), \boxdot ([x,M]) = [y, {\bf let} \bigcirc x = y {\: \bf in \:} M] \in \operatorname{Hom}_{\mathcal{C}(\lambda)}(\boxdot \hat{\varphi}, \boxdot \hat{\psi})$.
The following functoriality condition might be easily checked with the reduction rules:
\begin{enumerate}
  \item $\boxdot(g \circ f) = \boxdot g \circ \boxdot f$;
  \item $\boxdot (id_{\hat{A}}) = id_{\boxdot \hat{A}}$.
\end{enumerate}

We define the following maps. $\eta: Id_{\mathcal{C}(\lambda)} \rightarrow {\bf F}$ such that for each $\hat{\varphi} \in \operatorname{Ob}_{\mathcal{C}(\lambda)}$ one has $\eta_{\hat{\varphi}} = [x, {\bf pure \:} x] \in \operatorname{Hom}_{\mathcal{C}(\lambda)}(\hat{A}, \boxdot \hat{A})$.
We express a monoidal transformation as $m_{\hat{\varphi}, \hat{\psi}} : {\bf F} \hat{\varphi} \times {\bf F} \hat{\psi} \to {\bf F} (\hat{\varphi} \times \hat{\psi})$ such that one has $m_{\hat{\varphi}, \hat{\psi}} = [p, {\bf let} \bigcirc x,y = \pi_1 p, \pi_2 p {\: \bf in \:} \langle x, y \rangle] \in \operatorname{Hom}_{\mathcal{C}(\lambda)}(\boxdot \hat{\varphi} \times \boxdot \hat{\psi}, \boxdot (\hat{\varphi} \times \hat{\psi}))$. Also we express  $u_{\mathds{1}}$ as $[\sqbullet, {\bf let} \bigcirc \underline{\quad} = \underline{\quad} {\: \bf in \:} \sqbullet]$.

${\bf F}$ is clearly a monoidal endofunctor. Let us check the required coherence conditions:

\begin{lemma}
  $ $

\begin{enumerate}
  \item ${\bf F}(f) \circ \alpha_{\varphi} = \alpha_{\beta} \circ f$
  \item $(m_{\hat{\varphi}, \hat{\psi}}) \circ (\alpha_{\varphi} \times \alpha_{\beta}) = \alpha_{{\varphi} \times \alpha_{\beta}}$
  \item $u_{\mathds{1}} = \eta_{\mathds{1}}$
\end{enumerate}
\end{lemma}

\begin{proof}
  $ $

\begin{enumerate}
\item

$\eta_{\hat{\psi}} \circ f = [y, {\bf pure \:} y] \circ [x, M] = [x, {\bf pure \:} y [y := M]] = [x, {\bf pure \:} M]$

From the other hand, one has:

$\begin{array}{lll}
&\boxdot f \circ \eta_{\hat{A}} = & \\
&[z, {\bf let} \bigcirc x = z {\: \bf in \:} M] \circ [x, {\bf pure \:} x] = [x, {\bf let} \bigcirc x = z {\: \bf in \:} M [z := {\bf pure \:} x]] = & \\
&[x, {\bf let} \bigcirc x = {\bf pure \:} x {\: \bf in \:} M] = [x, {\bf pure \:} M [x := x]] = [x, {\bf pure \:} M]&
\end{array}$

\item $m_{\hat{A},\hat{B}} \circ (\eta_{\hat{A}} \times \eta_{\hat{B}}) = \eta_{\hat{A} \times \hat{B}}$ \\

$\begin{array}{lll}
& m_{\hat{A},\hat{B}} \circ (\eta_{\hat{A}} \times \eta_{\hat{B}}) = & \\
& [q, {\bf let \:} \bigcirc x, y = \pi_1 q, \pi_2 q {\: \bf in \:} \langle x, y \rangle] \circ [p, \langle {\bf pure \:} (\pi_1 p), {\bf pure \:} (\pi_2 p) \rangle] = & \\
& [p, {\bf let} \bigcirc x, y = \pi_1 q, \pi_2 q {\: \bf in \:} \langle x, y \rangle [q := \langle {\bf pure \:} (\pi_1 p), {\bf pure \:} (\pi_2 p) \rangle]] = & \\
& [p, {\bf let} \bigcirc x, y = \pi_1 (\langle {\bf pure \:} (\pi_1 p), {\bf pure \:} (\pi_2 p) \rangle), \pi_2 (\langle {\bf pure \:} (\pi_1 p), {\bf pure \:} (\pi_2 p) \rangle) {\: \bf in \:} \langle x, y \rangle] = & \\
& [p, {\bf let} \bigcirc x, y = {\bf pure \:} (\pi_1 p), {\bf pure \:} (\pi_2 p) {\: \bf in \:} \langle x, y \rangle] = & \\
& [p, {\bf pure \:} (\langle x,y \rangle [x := \pi_1 p, y := \pi_2 p])] = [p, {\bf pure \:} \langle \pi_1 p, \pi_2 p \rangle] = [p, {\bf pure \:} p] = \eta_{\hat{A} \times \hat{B}}&
\end{array}$
\item Immediately.
\end{enumerate}
\end{proof}

The previous results imply completeness.

\begin{lemma}

  $\langle \mathcal{C}(\lambda), \boxdot, \eta \rangle$ is an ${\bf IEL}^{-}$-category
\end{lemma}

\section{Cover semantics of the predicate extention of ${\bf IEL}^{-}$ and its relatives}

In this section, we discuss Heyting algebras with operators and their representations with cover systems. These operators reflect the ${\bf IEL}$-like modalities.

Generally, we often study predicate extensions of modal or intuitionistic logics using Kripke frames with domains and their generalisations such as Kripke bundles or simplicial frames. See this monograph by Gabbay, Shehtman and Skrotsov for the comprehensive discussion \cite{gabbay2009quantification}.

It is difficult to claim that predicate extensions of intuitionistic modal logics are studied well. The possible reason is that intuitionistic modal predicate Kripke frames are quite far to be intuitively clear. Goldblatt proposed the alternative approach in studying predicate extensions of intuitionistic modal (and even substructural) logics based on Kripke-Joyal semantics, see \cite{goldblatt2006kripke} \cite{goldblatt2011cover}. This approach is based on the Dedekind-MacNeille completion and the representation of locales with cover systems. We discuss below the motivation of cover semantics. Goldblatt and Hodkinson also proposed a somewhat kindred approach for classical modal predicate logic using models with admissible sets \cite{goldblatt2009commutativity}.

First, we discuss the relevant background on Heyting algebras and locales.

A Heyting algebra is a bounded distributive lattice $\mathcal{H} = \langle H, \wedge, \vee, \bot, \top \rangle$ with the binary operation $\Rightarrow$ such that the following equivalence holds:

\begin{center}
  $a \wedge b \leq c$ iff $a \leq b \Rightarrow c$
\end{center}

Recall that a \emph{locale} is a complete lattice $\mathcal{L} = \langle L, \wedge, \bigvee \rangle$ such that finite infima distribute over arbitrary suprema:

\begin{center}
  $a \wedge \bigvee B = \bigvee \{ a \wedge b \: | \: b \in B \}$ for each $B \subseteq L$.
\end{center}

The notion of a locale coincides with the notion of a complete Heyting algebra since the implication is defined uniquely for each $a, b \in L$ as follows:

\begin{center}
 $a \Rightarrow b = \bigvee \{ c \in \mathcal{L} \: | \: a \wedge c \leq b \}$
\end{center}

A locale is a central object in point-free topology, where a locale is a lattice-theoretic counterpart of a topological space. The aim of this discipline is to study topology considering spaces only with the structure of their topologies as lattices of opens mentioning no points. For the further discussion, see \cite{johnstone1982stone} \cite{joyal1984extension} \cite{picado2011frames}. In usual point-set topology, we are often interested in subspaces. In point-free topology, subspaces are characterised with operators on a locale called \emph{nuclei}. A nucleus on a Heyting algebra is a multiplicative closure operator.
\begin{defin}
  A nucleus on a Heyting algebra $\mathcal{H}$ is a monotone map $j : \mathcal{H} \to\mathcal{H}$ such that the following hold for all $a, b \in \mathcal{H}$:
  \begin{enumerate}
    \item $a \leq j a$
    \item $j a = j j a$
    \item $j (a \wedge b) = j a \wedge j b$
  \end{enumerate}
\end{defin}

One may consider a nucleus operator as a lattice-theoretic analogue of a Lawvere-Tierney topology that generalises the notion of a Grothendieck topology for a presheaf topos. In its turn, Lawvere-Tierney topology contains a modal operator often called a geometric modality \cite{lawvere1970quantifiers}. Here, one may read $j \varphi$ as ``it is locally the case that $\varphi$''. The logic of Heyting algebras with a nucleus operator was studied by Goldblatt from Kripkean and topos-theoretic perspectives, see \cite{goldblatt1981grothendieck}.

It is also well-known that the set of fixpoints of a nucleus on a Heyting algebra is a Heyting subalgebra. From a point-free topological view, nuclei characterise sublocales \cite{picado2011frames}. Those operators play a tremendous role in a locale representation as well. In this monograph \cite{dragalin1988mathematical}, Dragalin showed that any complete Heyting algebra is isomorphic to the locale of fixpoints of a nucleus operator on the algebra of up-sets. Moreover, any spatial locale (the lattice of open sets) is isomorphic to the complete Heyting algebra of fixpoints of a nucleus operator generated by some Dragalin frame. We recall that a Dragalin frame is a structure that generalises both Kripke and Beth semantics of intuitionistic logic. Bezhanishvili and Holliday strengthened this result for arbitrary complete Heyting algebras, see \cite{bezhanishvili2016locales}.

Goldblatt provided the alternative way of the representation of locales \cite{goldblatt2011cover} (and even arbitrary lattices \cite{goldblatt2017representing}) with cover systems. Dragalin frames and Goldblatt cover systems might be connected with each other, but it seems that the relationship between them is not investigated yet.

We examine Goldblatt's framework closely. First of all, let us recall some relevant notions.

Let $\langle P, \leq \rangle$ be a poset. A subset $A \subseteq P$ is called \emph{upwardly closed}, if $x \in A$ and $x \leq y$ implies $y \in A$. For $A \subseteq P$, ${\uparrow A} = \{ x \in P \: | \: \exists y \in A \: y \leq x \}$.
If $x \in P$, then the \emph{cone at} $x$ is an up-set ${\uparrow x} = {\uparrow \{ x \}}$. A subset $Y \subseteq P$ refines a subset $X \subseteq P$ if $Y \subseteq {\uparrow X}$. By $\operatorname{Up}(P, \leq)$ we will mean the poset (in fact, the locale) of all upwardly closed subsets of a partial order $\langle P, \leq \rangle$. It is also clear that the set of all upwardly closed sets forms a locale.

Here we consider triples $\mathcal{S} = \langle P, \leq, \triangleright \rangle$, where $\langle P, \leq \rangle$ is a poset and $\triangleright$ is a binary relation between $P$ and $\mathcal{P}(P)$. Given $x \in P$ and $C \subseteq P$, then we say that $x$ is \emph{covered} by $C$ ($C$ is an $x$-cover), if $x \triangleright C$ ($C \triangleleft x$). Cover systems were introduced to study local truth that comes from topological and topos-theoretic intuitions. A statement is locally true concerning some object as topological space or an open subset if this object has an open cover in each member of which the statement is true. For instance, such a statement might be local equality of continuous maps \cite{goldblatt1981grothendieck}. An abstract cover system has the following definition:

\begin{defin} A triple $\mathcal{S} = \langle P, \leq, \triangleright \rangle$ as above is called cover system, if the following axioms hold for $x \in P$:

\begin{enumerate}
  \item (Existence) There exists an $x$-cover $C$ such that $C \subseteq {\uparrow x}$
  \item (Transitivity) Let $x \triangleright C$ and for each $y \in C$ $y \triangleright C_y$, then $x \triangleright \bigcup_{y \in C} C_y$
  \item (Refinement) If $x \leq y$, then any $x$-cover might be refined to a $y$-cover. That is, $C \triangleleft x$ implies that there exists an $y$-cover $C'$ such that $C' \subseteq {\uparrow C}$
\end{enumerate}
\end{defin}

Let $\mathcal{S}$ be a cover system, let us define an operator $j : \mathcal{P}(P) \to \mathcal{P}(P)$ as

\begin{center}
  $j X = \{ x \in P \: | \: \exists C \: x \triangleright C \subseteq X \}$
\end{center}

If $x \in j X$ is called a local member of $X$. A subset $X \subseteq P$ is called localised if $j X \subseteq X$. A localised up-set is called a \emph{proposition}. $\operatorname{Prop}(\mathcal{S})$ is the set of all propositions of a cover system. Goldblatt showed that such an operator is a closure operator on a locale of all up-sets \cite{goldblatt2011cover} that follows from the axioms of a cover system. According to that, a subset $X$ is a proposition iff $X = {\uparrow X} = j X$.

\begin{defin} A cover system is called localic, if the following axiom hold
  \begin{center}
    Every $x$-cover can be refined to an $x$-cover included in ${\uparrow x}$.
  \end{center}
  That is, $x \triangleright C$ implies that there exists $x \triangleright C'$ such that $C' \subseteq {\uparrow C}$ and $C' \subseteq {\uparrow x}$.
\end{defin}

This localic axiom makes such a $j$-operator a nucleus. That is, if $\mathcal{S} = \langle P, \leq, \triangleright \rangle$ is a localic cover system, then $\operatorname{Prop}(\mathcal{S})$ is a sublocale of $\operatorname{Up}(P, \leq)$ since the set of fixpoints of nucleus is a sublocale of $\operatorname{Up}(P, \leq)$. Here we strengthen the fourth axiom of a localic cover system as:

\begin{center}
  Every $x$-cover is included in ${\uparrow x}$.
\end{center}

Such a local cover system is called a \emph{strictly localic cover system}. The fourth axiom is built in such generalisations of open cover systems as Grothendieck topology and cover schemes, see \cite{bell2003cover} \cite{maclane2012sheaves}.

The representation theorem for an abritrary locale is the following one \cite{goldblatt2011cover}.

\begin{theorem} \label{LocalIso}
  Let $\mathcal{L}$ be a locale, then there exists a strictly localic cover system $\mathcal{S}$ such that
  $\mathcal{L} \cong \operatorname{Prop}(\mathcal{S})$.
\end{theorem}

\begin{proof} We provide a proof sketch in order to remain the paper self-contained.

Given a locale $\mathcal{L} = \langle L, \bigvee, \wedge \rangle$. Let us define $\mathcal{S}_{\mathcal{L}} = \langle L, \sqsubseteq, \triangleright \rangle$ such that $x \sqsubseteq y$ iff $y \leq x$ and $x \triangleright C$ iff $x = \bigvee C$ in $\mathcal{L}$.
Then $\mathcal{S}_{\mathcal{L}}$ is a localic cover system. The strictness follows from the fact that if $x \triangleright C$, that is, $x = \bigvee C$, then $C \subseteq (x] = \{ y \: | \: y \sqsubseteq x \}$. Every cone $(x] = {\uparrow x}$ is localised, thus, $(x]$ is a proposition. It is not to so difficult to see that an arbitrary proposition of $\mathcal{S}_{\mathcal{L}}$ is a downset of $\sqsubseteq$.

  The map $x \mapsto (x]$ is an isomorphism.
\end{proof}

As a consequence, one has a uniform embedding for arbitrary Heyting algebras as follows.

Given a bounded lattice $\mathcal{L}$, a \emph{completion} of $\mathcal{L}$ is a complete lattice $\overline{\mathcal{L}}$ that contains $\mathcal{L}$ as a sublattice. A completion $\overline{\mathcal{L}}$ is called \emph{Dedekind-MacNeille} if every element of $a \in \overline{\mathcal{L}}$ is both a join and meet of elements of $\mathcal{L}$ (see \cite{davey2002introduction} to read more about lattice completions):
\begin{center}
  $a = \bigvee \{ b \in \mathcal{L} \: | \: a \leq b \} = \bigwedge \{ b \in \mathcal{L} \: | \: b \leq a \}$.
\end{center}
The class of all Heyting algebras is closed under Dedekind-MacNeille completions: if $\mathcal{H}$ is a Heyting algebra, then $\overline{\mathcal{H}}$ is a locale. The implication in an arbitrary Heyting algebra $\mathcal{H}$ has an extension as follows, where $a, b \in \overline{\mathcal{H}}$ \footnote{Completions of Heyting algebras are interesting topic itself, we refer the reader to this paper \cite{harding2004macneille} for further discussion.}:
\begin{center}
  $a \Rightarrow b = \bigwedge \{ c \rightarrow d \: | \: a \geq c \in \mathcal{H} \: \& \: d \leq b \in \mathcal{H} \}$.
\end{center}

By the previous theorem, $\overline{\mathcal{H}}$ is isomorphic to the locale of propositions of a strictly localic cover system $\mathcal{S}_{\overline{\mathcal{H}}}$. That is, we have the following fact.

\begin{theorem}
  Every Heyting algebra is isomorphic to a subalgebra of propositions of a suitable strictly localic cover system.
\end{theorem}

Strictly localic cover systems provide alternative model structures for intuitionistic predicate logic.
Let $\mathcal{S} = \langle P, \leq, \triangleright \rangle$ be a strictly localic cover system and let $D$ be a non-empty set, a domain of individuals. Let $V$ be a valuation function that maps each $k$-ary predicate letter $P$ to $V(P) : D^{k} \to \operatorname{Prop}(\mathcal{S})$.

To interpret variables, we use $D$-assignments that have the form of infinite sequences $\sigma = \langle \sigma_0, \sigma_1, \dots, \sigma_n, \dots \rangle$, where $\sigma_i \in D$ for each $i < \omega$. A $D$-assignment maps each variable $x_i$ to the corresponding $\sigma_i$. Given an assignment $\sigma$ and $d \in D$, then $\sigma(d / n)$ is a $D$-assignment obtained from $\sigma$ replacing $\sigma_n$ to $d$.

By $\operatorname{IPL}$-model we will mean a structure $\mathfrak{M} = \langle \mathcal{S}, D, V \rangle$, where $\mathcal{S}$ is a strictly localic cover system, $D$ is a domain of individuals, and $V$ is a $D$-valuation. Given a $D$-assignment and $x \in \mathcal{S}$, the truth relation $\mathfrak{S}, x, \sigma \models \varphi$ is defined inductively:

\begin{enumerate}
  \item $\mathfrak{M}, x, \sigma \Vdash P (x_{n_1}, \dots, x_{n_k})$ iff $x \in V(P)(\sigma_{n_1}, \dots, \sigma_{n_k})$.
  \item $\mathfrak{M}, x, \sigma \Vdash \bot$ iff $x \triangleright \emptyset$
  \item $\mathfrak{M}, x, \sigma \Vdash \varphi \land \psi$ iff $\mathfrak{M}, x, \sigma \Vdash \varphi$ and $\mathfrak{M}, x, \sigma \Vdash \psi$.
  \item $\mathfrak{M}, x, \sigma \Vdash \varphi \lor \psi$ iff there exists an $x$-cover $C$ such that for each $y \in C$ $\mathfrak{M}, y, \sigma \Vdash \varphi$ or $\mathfrak{M}, y, \sigma \Vdash \psi$.
  \item $\mathfrak{M}, x, \sigma \Vdash \varphi \rightarrow \psi$ iff for all $y \in {\uparrow x}$, if $\mathfrak{M}, y, \sigma \Vdash \varphi$ implies $\mathfrak{M}, y, \Vdash \psi$.
  \item $\mathfrak{M}, x, \sigma \Vdash \forall x_n \varphi$ iff for each $d \in D$, $\mathfrak{M}, x, \sigma(d / n) \Vdash \varphi$.
  \item $\mathfrak{M}, x, \sigma \Vdash \exists x_n \varphi$ iff there exist an $x$-cover $C$ and $d \in D$ such that for each $y \in C$ one has $\mathfrak{M}, y, \sigma(d/n) \Vdash \varphi$.
\end{enumerate}

Given a formula $\varphi$, one may associate a truth set $||\varphi||^{\mathfrak{M}}_{\sigma}$ defined with the locale operations on $\operatorname{Prop}(\mathcal{S})$:

\begin{enumerate}
  \item $||P (x_{n_1}, \dots, x_{n_k})||^{\mathfrak{M}}_{\sigma} = V(P)(\sigma_{n_1}, \dots, \sigma_{n_k})$
  \item $||\bot||^{\mathfrak{M}}_{\sigma} = j \emptyset$
  \item $||\varphi \land \psi||^{\mathfrak{M}}_{\sigma} = ||\varphi||^{\mathfrak{M}}_{\sigma} \cap ||\psi||^{\mathfrak{M}}_{\sigma}$
  \item $||\varphi \lor \psi||^{\mathfrak{M}}_{\sigma} = j(||\varphi||^{\mathfrak{M}}_{\sigma} \cup ||\psi||^{\mathfrak{M}}_{\sigma})$
  \item $||\varphi \rightarrow \psi||^{\mathfrak{M}}_{\sigma} = ||\varphi||^{\mathfrak{M}}_{\sigma} \Rightarrow ||\psi||^{\mathfrak{M}}_{\sigma}$
  \item $||\forall x_n \varphi||^{\mathfrak{M}}_{\sigma} = \bigwedge \limits_{d \in D} ||\varphi||^{\mathfrak{M}}_{\sigma(d/n)}$
  \item $||\exists x_n \varphi||^{\mathfrak{M}}_{\sigma} = j (\bigvee \limits_{d \in D} ||\varphi||^{\mathfrak{M}}_{\sigma(d/n)})$
\end{enumerate}
where $j$ is the associated nucleus on the locale of $\mathcal{S}$-propositions.

Thus, one has the completeness theorem \footnote{Here we note that this constuction admits generalisations and provides complete semantics for predicate substructural logics, see, e. g., \cite{goldblatt2006kripke} \cite{goldblatt2011grishin}.}:

\begin{theorem}
  Intuitionistic first-order logic is sound and complete with respect to $IPL$-models.
\end{theorem}

We define modal cover systems. Suppose one has a localic cover system $\mathcal{S} = \langle S, \leq, \triangleright \rangle$. We seek to extend $\mathcal{S}$ with a binary relation $R$ on $S$ that yields an operator on $\mathcal{P}(S)$:

\begin{center}
$\langle R \rangle A = \{ x \in S \: | \: \exists y \in A \: x R y\} = R^{-1}(A)$
\end{center}

\begin{defin}
  A quadruple $\mathcal{M} = \langle S, \leq, \triangleright, R \rangle$ is called a modal cover system, if a triple $\langle S, \leq, \triangleright \rangle$ is a strictly localic cover system and the following conditions hold:

  \begin{enumerate}
    \item (Confluence) If $x \leq y$ and $x R z$, then there exists $w$ such that $y R w$ and $z \leq w$.
    \item (Modal localisation) If there exists $C$ such that $x \triangleright C \subseteq \langle R \rangle A$, then there exists $y \in R(x)$ with a $y$-cover included in $X$.
  \end{enumerate}
\end{defin}

The first condition is a general requirement for intuitionistic modal logic allowing $\langle R \rangle A$ to be an up-set whenever $A$ is. The modal localisation principle claims that $\operatorname{Prop}(\mathcal{M})$ is closed under $\langle R \rangle$.

There is a representation theorem for locales with monotone operators, see \cite{goldblatt2011cover} to have a proof in detail:
\begin{theorem} \label{ModalIso}
  Let $\mathcal{L}$ be a locale and $m : \mathcal{L} \to \mathcal{L}$ a monotone map on $\mathcal{L}$, then the algebra $\langle \mathcal{L}, m \rangle$ is isomorphic to the algebra $\langle \operatorname{Prop}(\mathcal{S}_{\mathcal{L}}), \langle R_m \rangle \rangle$
\end{theorem}

\begin{proof}
  As we already know by Theorem~\ref{LocalIso}, $\mathcal{L} = \langle L, \bigvee, \wedge \rangle$ is isomorphic to the locale $\operatorname{Prop}(\mathcal{S}_{\mathcal{L}})$.

   $\mathcal{S}_{\mathcal{L}} = \langle L, \sqsubseteq, \triangleright \rangle$ is a strictly localic cover system, where $x \sqsubseteq y$ iff $y \leq x$ and $x \triangleright C$ iff $x = \bigvee C$ in $\mathcal{L}$. We recall that this isomorphism was established with map $x \mapsto (x] = \{ y \in L | y \sqsubseteq x\}$. Let us put $x R_m y$ iff $x \leq m y$. The relation is well-defined and the confluence and modal localisation conditions hold.

   The key observation is that $(m a] = \langle R_m \rangle (a]$.
\end{proof}

Goldblatt introduced modal cover systems to provide semantics for quantified lax logic and intuitionistic counterparts of the modal predicate logics ${\bf K}$ and ${\bf S}4$ \cite{goldblatt2011cover}. In the next subsection, we introduce similar cover systems to provide complete semantics for intuitionistic predicate modal logics with ${\bf IEL}^{-}$-like modalities.

\subsection{Prenuclei operators}

We discuss prenuclei operators, overview their use cases and provide representation for Heyting algebras with such operators via suitable modal localic cover systems. A weaker version of nuclei operators is quite helpful in point-free topology as well.

\begin{defin}
  $ $

  Let $\mathcal{H}$ be a Heyting algebra, a prenucleus on $\mathcal{H}$ is an operator
  monotone $j : \mathcal{H} \to \mathcal{H}$ such that for each $a, b \in \mathcal{H}$:

\begin{enumerate}
  \item $a \leq j a$
  \item $j a \wedge b \leq j (a \wedge b)$.
\end{enumerate}

A prenucleus is called multiplicative if it distributives over finite infima.
\end{defin}

By \emph{prenuclear algebra}, we will mean a pair $\langle \mathcal{H}, j \rangle$, where $j$ is a prenucleus on $\mathcal{H}$. A prenuclear algebra is localic if its Heyting reduct is a locale. A prenuclear algebra is \emph{multiplitcative} if its prenucleus is. Simmons calls multiplicative prenuclei merely as prenuclei \cite{simmons2010curious}, but this term is more spread for operators as defined above, see, e.g. \cite{picado2011frames}. We introduce the term ``multiplicative prenucleus'' in order to distinguish all those operators from each other since we consider both.

The use case of prenuclei operators is factorising locales considering sublocales as quotients. See the monograph by Picado and Pultr \cite{picado2011frames} for the discussion in detail. We just note that one may generate a nucleus by generating a sequence of prenuclei parametrised over ordinals.

One may involve multiplicative prenuclei to the study lattices of nuclei on a locale. Infima are defined pointwise there. Joins are more awkward to be described explicitly. Multiplicative prenuclei provide a suitable description of joins of nuclei in such locales. Multiplicative prenuclei form a complete Heyting algebra since they are closed under composition and pointwise directed joins. Thus, one may define joins of nuclei with so-called nuclear reflection, an approximation of a nucleus with prenuclei. Here we refer the reader to this paper \cite{escardo2003joins}, where this aspect has a more comprehensive explanation. The other aspect of multiplicative prenuclei were studied by Simmons \cite{simmons2010curious}.

Let us define a prenuclear cover system to have a suitable representation for prenuclear algebras.

\begin{defin} Let $\mathcal{S} = \langle S, \preceq, \triangleright, R \rangle$ be a modal cover system, then $\mathcal{S}$ is called prenuclear, if the following two conditions hold:

  \begin{enumerate}
    \item $R$ is reflexive.
    \item Let $x, y \in S$ such that $x R y$, then there exists $z \in {\uparrow y}$ such that $x \leq z$ and $x \in R(z)$.
  \end{enumerate}
\end{defin}

One may visualise the second condition with the following diagram:

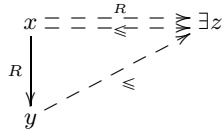
\xymatrix{
&&&&&& x \ar[d]_{R} \ar@<-.5ex>@{-->}[rr]^{ ^R} \ar@<.5ex>@{-->}[rr]_{ \leq} && \exists z \\
&&&&&& y \ar@{-->}[urr]_{\leq}
}

This lemma claims that a prenuclear cover system is well-defined as follows, the similar statement was proved by Goldblatt for nuclear cover systems \cite{goldblatt2011cover}:

\begin{lemma} \label{prenucleuslemma}
  Let $\mathcal{S} = \langle P, \leq, \triangleright, R \rangle$ be a prenuclear cover system, then $\langle R \rangle$ is a prenucleus on $\operatorname{Prop}(\mathcal{S})$, that is for each $A, B \in \operatorname{Prop}(\mathcal{S})$:

  \begin{enumerate}
    \item $A \subseteq \langle R \rangle A$
    \item $A \cap \langle R \rangle B \subseteq \langle R \rangle (A \cap B)$
  \end{enumerate}
\end{lemma}

\begin{proof}
The condition $A \subseteq \langle R \rangle A$ holds according to the standard modal logic argument.

Let us check the second condition. Let $A \cap \langle R \rangle B$, then $x \in A$ and $x \in R(y)$ for some $y \in B$. $x R y$ implies there exists $z \in {\uparrow y}$ such that $x R z$ and $x \leq z$. $A$ is an up-set, then $z \in A$, so $z \in A \cap B$, but $x R z$, thus, $x \in \langle R \rangle (A \cap B)$.
\end{proof}

The lemma above allows one to extend the representation of arbitrary modal cover system described in the proof of Theorem~\ref{ModalIso} to prenuclear ones:

\begin{theorem} \label{prenuclearrepresentation}
  Every localic prenuclear algebra is isomorphic to the algebra of propositions associated with some modal prenuclear localic cover system.
\end{theorem}

\begin{proof}
Let $\mathcal{L} = \langle L, \bigvee, \wedge \rangle$ be a locale and $\mathfrak{L} = \langle \mathcal{L}, \iota \rangle$ a localic prenuclear algebra. Then $\mathcal{S}_{\mathfrak{L}} = \langle L, \sqsubseteq, \triangleright, R_{\iota} \rangle$ is a modal cover system, where $x R_{\iota} y$ iff $x \leq \iota y$.

Let us ensure that this cover system is prenuclear one. The relation is clearly reflexive, $x R_{\iota} x$ follows from the inflationary condition. The second prenuclear cover system axiom also holds. $x R_{\iota} y$, then $x \leq \iota y$. Let us put $z = x \wedge y$, then $x R_{\iota} z$ since $x \leq x \wedge \iota y \leq \iota (x \wedge y)$. $y \sqsubseteq z$ holds obviously.
\end{proof}

For extend the representation we discussed above, one needs to preserve prenuclei under the Dedekind-MacNeille completion.

Given a lattice $\mathcal{L}$ and $f : \mathcal{L} \to \mathcal{L}$ a monotone function on this lattice, let us define maps $f^{\circ}, f^{\bullet} : \overline{\mathcal{L}} \to \overline{\mathcal{L}}$ for $a \in \overline{\mathcal{L}}$:
\begin{center}
  $f^{\circ}(a) = \bigvee \{ f(x) \: | \: a \geq x \in \mathcal{L} \}$

  $f^{\bullet}(a) = \bigwedge \{ f(x) \: | \: a \leq x \in \mathcal{L} \}$
\end{center}
$f^{\circ}$ and $f^{\bullet}$ both extend $f$ and $f^{\circ} \leq f^{\bullet}$. Generally, neither $f^{\circ}$ is multiplicative nor $f^{\bullet}$, if $f$ is. One the other hand, if $f$ is a multiplicative function on a Heyting algebra, so is $f^{\circ}$, see \cite{theunissen2007macneille}.

\begin{lemma} \label{PrenucMac}
  Let $\iota$ be a prenucleus on a Heyting algebra $\mathcal{H}$, then $\iota^{\bullet}$ is a prenucleus on $\overline{\mathcal{H}}$.
\end{lemma}

\begin{proof}
  The proof is similar for the analogous statement about nuclei \cite{goldblatt2011cover}.
  $\iota$ is inflationary, so is $\iota^{\bullet}$, it is readily checked.
  Let us check that $a \wedge \iota^{\bullet} b \leq \iota^{\bullet}(a \wedge b)$ for each $a, b \in \overline{\mathcal{H}}$.
  It is known that, if $f, g$ are monotone functions on $\mathcal{H}$ with
  $f(a) \wedge g(b) \leq g(a \wedge b)$ for each $a, b \in \mathcal{H}$, then
  $f^{\circ}(x) \wedge g^{\bullet}(y) \leq g^{\bullet}(x \wedge y)$ for each $x, y \in \overline{\mathcal{H}}$ \cite{goldblatt2011cover}. We apply this statement putting the identity function as $f$ and $\iota$ as $g$.
\end{proof}

One may prove the following representation theorem for Heyting algebra with prenuclei operators combining Theorem~\ref{prenuclearrepresentation} and Lemma~\ref{PrenucMac}:

\begin{theorem} \label{PrenucRepr}
Every prenuclear algebra is isomorphic to the algebra of propositions obtained by some prenuclear localic cover system.
\end{theorem}

We consider the multiplicative case. Lower extensions respect multiplicativity and upper ones preserve inflationarity. We provide the equivalent definition of a multiplicative prenuclear algebra as follows to simplify the issue:

\begin{prop} \label{AltMult}
  Let $\mathcal{H}$ be a Heyting algebra and $j$ a monotone function that preserves finite infima, then the following are equivalent for each $a, b \in \mathcal{H}$:
\begin{enumerate}
  \item $a \leq j a$,
  \item $a \wedge j b \leq j (a \wedge b)$.
\end{enumerate}
\end{prop}

\begin{proof}
  Both implications are quite simple.
  One has $a = a \wedge \top = a \wedge j \top \leq j (a \wedge \top) = j a$.
  On the other hand, $a \wedge j b \leq j a \wedge j b = j (a \wedge b)$.
\end{proof}

\begin{lemma} \label{MacNeilleMult}
  Let $\mathcal{H}$ be a Heyting algebra and $\iota$ a multiplicative prenucleus on $\mathcal{H}$, then $\iota^{\circ}$ is a multiplicative nucleus on $\overline{\mathcal{H}}$.
\end{lemma}

\begin{proof}
According to Proposition~\ref{AltMult}, one may equivalently replace the inflationarity condition to $a \wedge \iota b \leq \iota (a \wedge b)$. In fact, one needs to check that the inequation
$x \wedge \iota^{\circ} y \leq \iota^{\circ} (x \wedge y)$ holds for each $x, y \in \overline{\mathcal{H}}$. One has:

\vspace{\baselineskip}

$\begin{array}{lll}
& x \wedge \iota^{\circ} y = \bigvee \{ a \in \mathcal{H} \: | \: a \leq x \} \wedge \bigvee \{ \iota b \in \mathcal{H} \: | \: b \in \mathcal{H}, b \leq y \} = & \\
\:\: & \bigvee \{ a \wedge \iota b \: | \: a \leq x, b \leq y, a, b \in \mathcal{H} \} \leq \bigvee \{ \iota(a \wedge b) \: | \: a \leq x, b \leq y, a, b \in \mathcal{H} \} \leq & \\
\:\: & \bigvee \{ \iota c \: | \: c \in \mathcal{H}, c \leq x \wedge y \} = \iota^{\circ} (x \wedge y)&
\end{array}$

\vspace{\baselineskip}

$\iota^{\circ}$ is multiplicative since $\iota$ is multiplicative. Thus, $\iota^{\circ}$ is a multiplicative prenucleus on $\mathcal{H}$.
\end{proof}

Let us define a suitable cover system.

\begin{defin}
  Let $\mathcal{M} = \langle S, \leq, \triangleright, R \rangle$ be a modal cover system, then $\mathcal{M}$ is called multiplicative prenuclear if the following conditions hold:
  \begin{enumerate}
    \item $R$ is serial, that is, for each $x \in S$ there exists $y \in S$ such that $x R y$.
    \item if $x R y$ and $x R z$ then there exists $w \in {\uparrow x} \cap {\uparrow y}$ such that $x R w$.
    \item Let $x, y \in S$ such that $x R y$, then there exists $z \in {\uparrow y}$ such that $x \leq z$ and $x \in R(z)$.
  \end{enumerate}
\end{defin}

One may consider a multiplicative prenuclear frame as an $R_{\Box}$-reduct of a $\operatorname{CK}$-modal cover system \cite{goldblatt2011cover} with the additional principle that corresponds to the second postulate of a prenuclear cover system. Such a cover system describes the logic with the modal axioms $\bigcirc \top$, $\bigcirc p \land \bigcirc q \rightarrow \bigcirc (p \land q)$, and $p \land \bigcirc q \rightarrow \bigcirc (p \land q)$ plus the $\bigcirc$-monotonicity rule. It is not so hard to see that this logic is deductively equivalent to ${\bf IEL}^{-}$ over intuitionistic logic.

\begin{lemma} \label{covermult}
  Let $\mathcal{M} = \langle S, \leq, \triangleright, R \rangle$ be a multiplicative prenuclear cover system, then $\langle R \rangle$ is a multiplicative prenucleus on $\operatorname{Prop}(\mathcal{S})$.
\end{lemma}

\begin{proof}
  $\langle R \rangle$ is clearly serial. The multiplicativity follows from the second postulate of a multiplicative prenuclear cover system. The third equation is proved similarly to Lemma~\ref{prenucleuslemma}.
\end{proof}

Theorem~\ref{prenuclearrepresentation}, Lemma~\ref{covermult}, and Lemma~\ref{MacNeilleMult} imply the following theorem.

\begin{theorem} \label{multRepr}
$ $

\begin{enumerate}
\item Every localic multiplicative prenuclear algebra is representable as a modal locale of the propositions obtained by a suitable modal cover system.
\item Every multiplicative prenuclear algerba is isomorphic to the subalgerba to the algebra of propositions obtained by some multiplicative prenuclear localic cover system.
\end{enumerate}
\end{theorem}

Finally, we consider ${\bf IEL}$-cover systems and corresponding multiplicative prenuclear algebras, where the equation $j \bot = \bot$ is satisfied. We call such multiplicative prenuclear algebras \emph{dense}
\footnote{The terminology is due to \cite{bezhanishvili2016locales}, where a nucleus with the similar condition is called dense.}. In particular, $j \bot = \bot$ implies $j^{\circ} \bot = \bot$ \cite{theunissen2007macneille}. Thus, if an operator on a Heyting algebra is a dense multiplicative prenucleus, so is its lower Dedekind-MacNeille completion.

An ${\bf IEL}$-cover system is a multiplicative prenuclear system $\mathcal{S} = \langle P, \leq, \triangleright, R \rangle$ such that for each $x, y \in S$ if $x R y$ and $y \triangleright \emptyset$ implies $x \triangleright \emptyset$. This condition yields $\langle R \rangle \emptyset = \emptyset$.

Thus, one may immediately extend Theorem~\ref{multRepr}:

\begin{theorem} \label{ielrepresentation}
  $ $

  \begin{enumerate}
    \item Every localic dense multiplicative prenuclear algebra is isomorphic to the algerba of propositions of some ${\bf IEL}$-cover system.
    \item Every dense multiplicative prenuclear algebra is isomorphic to the subalgebra of propositions of some ${\bf IEL}$-cover system.
  \end{enumerate}
\end{theorem}

\subsection{Completeness theorems}

In this subsection, by ${\bf IEL}^{-}_{-}$ we mean the set of formulas that contains all theorems of intutionistic propositional logic, formulas $ \varphi \to \bigcirc \varphi$ and $\varphi \land \bigcirc \psi \to \bigcirc (\varphi \land \psi)$, and is closed under Modus Ponens, Substitution, and $\bigcirc$-monotonicity.

Let us define first the intuitionistic modal predicate logic ${\bf QIEL}^{-}_{-}$ as an extension of intuitionistic predicate logic with the modal axioms that correspond to the conditions of a prenucleus operator. We consider a signature consisting of predicate symbols of an arbitrary arity lacking function symbols and individual constants.

\begin{enumerate}
  \item ${\bf IEL}^{-}_{-}$-axioms
  \item $\forall x \varphi \to \varphi (t/x)$
  \item $\varphi (t/x) \to \exists x \varphi$
  \item The inference rules are Modus Ponens, Bernays rules, and $\bigcirc$-monotonicity.
\end{enumerate}

Then ${\bf QIEL}^{-} = {\bf QIEL}^{-}_{-} \oplus \bigcirc (\varphi \to \psi) \to (\bigcirc \varphi \to \bigcirc \psi)$ and ${\bf QIEL} = {\bf QIEL}^{-} \oplus \neg \bigcirc \bot$, where $\neg \varphi = \varphi \to \bot$.

In this section we show that the logics above are complete with respect to their suitable cover systems. Let $\mathcal{L}$ be a logic above ${\bf QIEL}^{-}$, let us define their models. Let $\mathcal{C}$ be a prenuclear cover system $\mathcal{M} = \langle S, \leq, \triangleright, R \rangle$, $V$ a valuation function, and $D$ a set of individuals, then an $\mathcal{L}$-cover model is a triple $\mathfrak{M} = \langle \mathcal{M}, V, D \rangle$. Given a $D$-assignment and $x \in S$, the modal operator has the following semantics:

\begin{center}
  $\mathfrak{M}, x, \sigma \models \bigcirc \varphi$ iff there exists $y \in R(x)$ such that $\mathcal{M}, y, \sigma \models \varphi$.
\end{center}

In contrast to Kripkean semantics of ${\bf IEL}$-like logics, we interpret modality in terms of ``possibility''. Indeed, one may reformulate the truth condition above by means of an $\langle R \rangle$-operator on the locale of propositions:

\begin{center}
  $||\bigcirc \varphi ||^{\mathfrak{M}}_{\sigma} = \langle R \rangle ||\varphi||^{\mathfrak{M}}_{\sigma}$
\end{center}
The completeness theorem converts the Lindenbaum-Tarksi algebra to a suitable locale with the relevant modal operator with the Dedekind-MacNeille completion. After that, we represent this algebra as an algebra of localised up-sets by the representation theorem. To be more precise, we have:

\begin{theorem}
  Let $\mathcal{L} \in \{ {\bf IEL}^{-}_{-}, {\bf IEL}^{-}, {\bf IEL} \}$, then ${\bf Q}\mathcal{L}$ is sound and complete with repsect to their cover systems.
\end{theorem}

\begin{proof}
  Let us consider the ${\bf QIEL}^{-}_{-}$-case, the rest two cases are shown in the same fashion with the relevant representation theorem and the Dedekind-MacNeille completion.

  Let $\operatorname{Fm}$ be the set of all formulas and $\mathcal{V}$ the countable set of all individual variables. One has an equivalence relation $\varphi \sim \psi$ ${{\bf QIEL}^{-}_{-}} \vdash \varphi \to \psi$ and ${{\bf QIEL}^{-}_{-}} \vdash \psi \to \varphi$.

  Then, one has an ordering on $\operatorname{Fm} / \sim$ defined as $[\varphi]_{\sim} \leq [\psi]_{\sim}$ iff $\vdash_{{\bf QIEL}^{-}_{-}} \varphi \to \psi$. The operations on $\operatorname{Fm} / \sim$ are defined as:

\begin{center}
$[\bot]_{\sim} = \bot$

$[\varphi \land \psi]_{\sim} = [\varphi]_{\sim} \wedge [\psi]_{\sim}$

$[\varphi \lor \psi]_{\sim} = [\varphi]_{\sim} \lor [\psi]_{\sim}$

$[\varphi \to \psi]_{\sim} = [\varphi]_{\sim} \Rightarrow [\psi]_{\sim}$

$[\forall x \varphi]_{\sim} = \bigwedge \limits_{x \in V} [\varphi]_{\sim}$

$[\exists x \varphi]_{\sim} = \bigvee \limits_{x \in V} [\varphi]_{\sim}$

$[\bigcirc \varphi]_{\sim} = \bigcirc [\varphi]_{\sim}$

$\top = [\varphi]_{\sim}$, where ${{\bf IEL}^{-}_{-}} \vdash \varphi$
\end{center}

This algebra is clearly prenuclear, but its Heyting reduct is not necessarily complete. By Lemma~\ref{PrenucRepr}, one may embed the Lindenbaum-Tarksi algebra $\mathcal{L}_{{\bf QIEL}^{-}_{-}}$ to the localic prenuclear algebra $\langle \overline{\mathcal{F} / \sim}, \bigcirc^{\bullet} \rangle$ by Theorem~\ref{PrenucRepr}. The algebra $\langle \overline{\mathcal{F} / \sim},  \bigcirc^{\bullet} \rangle$ is isomorphic to the algebra of propositions of some prenuclear cover system by Theorem~\ref{prenuclearrepresentation}.

Thus, one has the isomorphism $f : \langle \overline{\mathcal{F} / \sim},  \bigcirc^{\bullet} \rangle \cong \langle \operatorname{Prop}(\mathcal{S}_{{\bf QIEL}^{-}_{-}}), \langle R_{\bigcirc }\rangle, \rangle$, where $\mathcal{S}_{{\bf QIEL}^{-}_{-}}$ is the obtained prenuclear cover system.

Let define a ${\bf QIEL}^{-}_{-}$ cover model $\mathfrak{M} = \langle \mathcal{S}_{{\bf QIEL}^{-}_{-}}, D, V\rangle$ putting $D = \operatorname{V}$. A valuation $V$ is defined as $V(P)(x_{n_1}, \dots, x_{n_k}) = f([P(x_{n_1}, \dots, x_{n_k})]_{\sim})$.
Here, a $D$-assignment $\sigma$ is merely an identity function.

Here, the key observation is $||\varphi||_{\sigma}^{\mathfrak{M}} = f [\varphi]_{\sim}$ that might be shown by easy induction on $\varphi$.

Then, if $\varphi$ is true in every ${\bf QIEL}^{-}_{-}$-model, then $||\varphi||_{\sigma}^{\mathfrak{M}} = \top = S$, thus, $f|\varphi| = \top$. Hence, ${\bf QIEL}^{-}_{-} \vdash \varphi$. Thus, ${\bf QIEL}^{-}_{-}$ is sound and complete with respect models on prenuclear cover systems.

The ${\bf QIEL}^{-}$ (${\bf QIEL}$) case follows from the same construction using Theorem~\ref{multRepr} (Theorem~\ref{ielrepresentation}).

\end{proof}

\section{Acknowledgements}

The author is grateful to Sergei Artemov, Lev Beklemishev, Neel Krishnaswami, Vladimir Krupski, Valerii Plisko, Anil Nerode, Ilya Shapirovsky, Valentin Shehtman, and Vladimir Vasyukov for consulting, helpful advice and conversations, feedback, and valuable suggestions. Some of the results were presented at the conference Logical Foundations of Computer Science 2020, Miami and the Computational Logic Seminar (The Graduate Centre, City University of New York).

This paper is dedicated in honour of the late Alexander Rogozin (1949 -- 2014).

The research is supported by RFBR grant 19-011-00799.

\bibliographystyle{plain}
\bibliography{Paper}

\end{document}